\documentclass[10pt]
{article}

\makeindex

\input xy
\xyoption{all}

\usepackage{mathtools}
\usepackage{amsmath,amsfonts,amssymb,amsthm}
\usepackage{latexsym, xspace, enumerate}
\usepackage[mathscr]{eucal}
\usepackage{hyperref}
\usepackage{amscd}
\usepackage{xcolor}
\newtheorem{theorem}{Theorem}[section]

\newtheorem{claim}[theorem]{Claim}

\newtheorem{proposition}[theorem]{Proposition}
\newtheorem{lemma}[theorem]{Lemma}

\newtheorem{corollary}[theorem]{Corollary}

\theoremstyle{definition}

\newtheorem{definition}[theorem]{Definition}
\newtheorem{question}[theorem]{Question}
\newtheorem{example}[theorem]{Example}
\newtheorem{remark}[theorem]{Remark}%[section]
\DeclareMathOperator{\supp}{\mathrm{supp}}

\newcommand{\N}{\mathbb{N}}
\newcommand{\T}{\mathbb{T}}

\newcommand{\Z}{\mathbb{Z}}
\newcommand{\Q}{\mathbb{Q}}
\newcommand{\R}{\mathbb{R}}
\newcommand{\I}{\mathbb{I}}

\newcommand{\eps}{\varepsilon}

\newcommand{\Tx}{\mathtt{T}_x}
\newcommand{\Ax}{\mathtt{A}_x}

\def\uu{\mathbf{u}}

\renewcommand{\I}{\mathcal I}
\newcommand{\F}{\mathcal F}

\newcommand{\A}{\mathcal A}

\renewcommand{\P}{\mathcal P}

\newcommand{\bx}{\mathrm{(b}_x\mathrm{)}}

\newcommand{\ax}{\mathrm{(a}_x\mathrm{)}}
\newcommand{\axuno}{\mathrm{(a1}_x\mathrm{)}}
\newcommand{\axdue}{\mathrm{(a2}_x\mathrm{)}}

\newcommand{\ix}{\mathrm{(i}_x\mathrm{)}}
\newcommand{\iix}{\mathrm{(ii}_x\mathrm{)}}
\newcommand{\iiix}{\mathrm{(iii}_x\mathrm{)}}

\newcommand{\unox}{\mathrm{(1}_x\mathrm{)}}

\newcommand{\duex}{\mathrm{(2}_x\mathrm{)}}

\newcommand{\trex}{\mathrm{(3}_x\mathrm{)}}

\newcommand{\Ix}{\mathrm{(I}_x\mathrm{)}}
\newcommand{\IIx}{\mathrm{(II}_x\mathrm{)}}

\newcommand{\stellaa}{\mathrm{(**)}}
\newcommand{\dageq}{\mathrm{(\dag)}}
\newcommand{\stareq}{\mathrm{(\star)}}

\numberwithin{equation}{section}

\newcommand{\NB}%{}
{${\clubsuit}\;$}
\newcommand{\NBA}%{}
{${\spadesuit}$}

\usepackage{vmargin}
\setmarginsrb{18mm}{5mm}{18mm}{10mm}%
        {5mm}{6mm}{5mm}{9mm}
\newlength{\bibitemsep}
\newlength{\bibparskip}
\let\oldthebibliography\thebibliography
\renewcommand\thebibliography[1]{%
  \oldthebibliography{#1}%
  \setlength{\parskip}{\bibitemsep}%
  \setlength{\itemsep}{\bibparskip}%
}

\usepackage{color}

\title{Nested ideals and topologically $\uu_\I$-torsion elements \\ of the circle group}

\date{\scriptsize Department of Mathematics, Computer Science and Physics - University of Udine}

\author{R. Di Santo\\ \href{mailto:raffaele.disanto@uniud.it}{\scriptsize raffaele.disanto@uniud.it}  
\and D. Dikranjan \\ \href{mailto:dikran.dikranjan@uniud.it}{\scriptsize dikran.dikranjan@uniud.it}
\and A. Giordano Bruno\footnote{The third named author is member of and thanks the National Group for Algebraic and Geometric Structures, and Their Applications (GNSAGA - INdAM)} \\ \href{mailto:anna.giordanobruno@uniud.it}{\scriptsize anna.giordanobruno@uniud.it}
\and H. Weber \\ \href{mailto:hans.weber@uniud.it}{\scriptsize hans.weber@uniud.it}}

\begin{document}
\maketitle

\begin{abstract}
Let $\uu=(u_n)_{n\in\N}$ be a sequence in $\N_+$ with $u_0=1$ and $u_n\mid u_{n+1}$ for every $n\in\N$, and let $b_n:=u_{n+1}/u_n$ for every $n\in\N_+$.
For every $r\in [0,1)$, there exists a unique sequence $(c_n)_{n\in\N_+}$ in $\N$ such that $r= \sum_{n=1}^\infty\frac{c_n}{u_n}$,
with $c_n<b_n$ for every $n\in\N_+$, and $c_n<b_n-1$ for infinitely many $n\in\N_+$; let $\supp(r):=\{n\in\N_+: c_n\neq0\}$ and $\supp_b(r):=\{n\in\N_+: c_n = b_n-1\}$. For $x=r+\Z\in \T$, let $\supp(x) = \supp(r)$ and $\supp_b(x) = \supp_b(r)$.

For an ideal $\I$ of $\N$, an element $x$ of the circle group $\T$ is called a \emph{topologically $\uu_\I$-torsion element of $\T$} if $u_nx$ $\I$-converges to $0$, that is, $\{n\in \N: u_nx \not \in U\}\in \I$ for every neighborhood $U$ of $0$ in $\T$. 
In this paper, under suitable conditions on the ideal $\I$, we completely describe the $\uu_\I$-torsion elements $x$ of $\T$ with $\lim_{n\in\supp(x)}b_n=\infty$ and those with $\{b_n:n\in\supp(x)\}$ bounded. 

According to \cite[Corollary 2.12]{Ghosh}, an element $x \in\T$ with $\{b_n:n\in\supp(x)\}$ bounded is topologically $\uu_\I$-torsion if and only if
$\supp(x)+1\setminus \supp(x)\in \I$ and $\supp(x) \setminus \supp_b(x) \in \I$. %\eqno(\dag)$$
We characterize the ideals $\I$ of $\N$, naming them  \emph{nested}, such that this equivalence holds 
%\NB\footnote{qui lascio non precisato che se vale per una b-bounded, allora vale anche per tutte.} 
and we provide examples of non-nested ideals $\I$ that satisfy the above mentioned  suitable conditions, 
so that the equivalence claimed in \cite[Corollary 2.12]{Ghosh} fails for those $\I$. 
%Si potrebbe aggiungere: 
For the ideal $\I_d$ of asymptotically dense sets in $\N$ the above described equivalence was anticipated in \cite[Corollary 3.9]{DG1}
(even if there is a gap in the proof). Since $\I_d$ is nested, our general result demonstrates, among others, that \cite[Corollary 3.9]{DG1} holds true.
%\footnote{era: Moreover, we introduce a special useful class of ideals called \emph{nested}, that in particular we use to find a counterexample to~\cite[Corollary 2.12]{Ghosh} and to show that  \cite[Corollary 3.9]{DG1} holds true, even if there is a gap in its original proof.\\
%Testo troppo vago da una parte, ma dice troppe cose (anche in un'altra direzione -- vedi \cite[Corollary 3.9]{DG1}). \\
%Ho scelto un'altra strada: raccontare un po' meglio ciò che abbiamo da dire su  \cite[Corollary 2.12]{Ghosh}, lasciando 
%la linea  \cite[Corollary 3.9]{DG1} per una lettura approfondita del testo. Volendo si può integrare anche la \cite[Corollary 3.9]{DG1} verso la fine, aggiungendo la 
%frase seguente, che qui lascio sotto \%\%\%. Volendo basta aprirla (viene bene, penso), ma non posso giurare
%che interessa moltissimo proprio tutti i potenziali lettori, sicuramente alcuni sì, ma altri no :-) \\
%\footnote{Si potrebbe aggiungere: For the ideal $\I_d$ of asymptotically dense sets in $\N$ the above described equivalence was anticipated in \cite[Corollary 3.9]{DG1}
%(even if there is a gap in the proof). Since $\I_d$ is nested, our general result demonstrates, among others, that \cite[Corollary 3.9]{DG1} holds true.}
% For the ideal $\I_d$ of asymptotically dense sets the above described equivalence was anticipated in \cite[Corollary 3.9]{DG1}, 
%even if there is a gap in their original proof. Since $\I_d$ is nested, our general result demonstrates, among others, that \cite[Corollary 3.9]{DG1} holds true.
\end{abstract}

\medskip
\hrule


{\scriptsize \tableofcontents}

\section{Introduction}

\subsection{Historical background}\label{Sec:Hist}

Let $\T=\R/\Z$ be the circle group denoted additively and let $\varphi\colon \R\to \T$ be the canonical projection; often we write $\bar x$ for  $\varphi(x)$ for the sake of  brevity. Given a sequence $\uu=(u_n)$ of integers, an element $x\in\T$ is \emph{topologically $\uu$-torsion} if $u_nx\to 0$ (see~\cite{Arm,Bra,D1,DPS,Ro,Vi}). Let $t_\uu(\T):=\{x\in\T: u_nx\to 0\}$, which is a subgroup of $\T$.
Following~\cite{BDS}, a subgroup $H$ of $\T$ is said to be \emph{characterized} ({\em by $\uu$}) if there exists a sequence of integers $\uu$ such that $H=t_\uu(\T)$. The introduction of these subgroups was motivated by problems arising from various areas of Mathematics, and they were thoroughly investigated (e.g., see the survey~\cite{DDG}). 

\smallskip
In this paper we are interested in the case of arithmetic sequences: recall that a sequence $\uu = (u_n)$ of integers is \emph{arithmetic} if it is strictly increasing, $u_0=1$ and $u_n\mid u_{n+1}$ for every $n\in\N$, and let $\A$ denote the family of all arithmetic sequences. 

For any $\uu\in\A$, let $b_n:={u_{n}}/{u_{n-1}}$ for every $n\in\N_+$ and $\mathbf b=(b_n)$.  
Following~\cite{DI}, call an infinite  subset $A$ of $\N$ \emph{$b$-bounded} (resp., \emph{$b$-divergent}) if the sequence $(b_n)_{n\in A}$ is bounded (resp., diverges to infinity).
For the sake of convenience we formally consider as {$b$-bounded} also all finite sets.
In particular, $\uu\in\A$ is \emph{$b$-bounded} if $\N_+$ is $b$-bounded, and $\uu$ is \emph{$b$-divergent} if $\N_+$ is $b$-divergent.

Each arithmetic sequence $\uu$ gives rise to a nice representation:
for every $x\in [0,1)$, there exists a unique sequence $(c_n(x))_{n\in\N_+}$ in $\N$ such that
\begin{equation}\label{ex-4}
{x= \sum_{n=1}^\infty\frac{c_n(x)}{u_n},}
\end{equation}
with $c_n(x)<b_n$ for every $n\in\N_+$, and $c_n(x)<b_n-1$ for infinitely many $n\in\N_+$.
When no confusion is possible, we shall simply write $c_n$ in place of $c_n(x)$.
For $x\in[0,1)$, with canonical representation~\eqref{ex-4}, let 
$\supp(x):=\{n\in\N_+: c_n\neq0\}$ and $\supp_b(x):=\{n\in\N_+: c_n=b_n-1\}$.

\smallskip
The following complete description of the $\uu$-torsion elements of $\T$ was given in~\cite[Theorem~2.3]{DI} (the slightly simplified version given below was recently proposed in~\cite{I-torsion}). 

\begin{theorem}[{\cite{I-torsion,DI}}]\label{DiD}
Let $\uu\in\mathcal A$ and $x\in[0,1)$ with $S=\supp(x)$ and $S_b=\supp_b(x)$. Then $\bar x \in t_{\uu}(\T)$ if and only if for all infinite $A\subseteq\N$ the following holds.
\begin{itemize}
   	\item[$\mathrm{(a)}$] If $A$ is $b$-bounded, then:
\begin{itemize}
	\item[$\mathrm{(a1)}$] if $A\subseteq^*S$,  then $A+1 \subseteq^*S$, $A\subseteq^*S_b$ and $\lim\limits_{n\in A}{\frac{c_{n+1}+1}{b_{n+1}}}=1$;
%\\ Moreover, if $A+1$ is $b$-bounded, then $A+1\subseteq^*S_b$ as well;
	\item[$\mathrm{(a2)}$] if $A\cap S$ is finite, then $\lim\limits_{n\in A}{\frac{c_{n+1}}{b_{n+1}}}=0$.
	%\\ Moreover, if $A+1$ is $b$-bounded, then $(A+1)\cap S$ is finite as well.
\end{itemize}
	\item[$\mathrm{(b)}$] If $A$ is $b$-divergent, then $\lim\limits_{n\in A}{\varphi\left(\frac{c_n}{b_n}\right)}=\lim\limits_{n\in A}{\varphi\left(\frac{c_n+1}{b_n}\right)}=0$.
\end{itemize}
\end{theorem}

 It is worth mentioning that a first result in this direction was given as an answer to a problem posed by Armacost \cite{Arm}; it was obtained independently by Borel~\cite{Bo2}, and somewhat earlier by  Dikranjan, Prodanov and Stoyanov~\cite{DPS}
%\NB\footnote{Questa frase in parentesi è (quasi!) un doppione di quello che si dice sotto;
%l'errore è uno solo (quello nel caso generale, che porta, ahimè, anche al corollario errato nel caso di $\T!$). Ma, oggettivamente, sono due errori e volendo si può anche dire due volte :-)
%} 
(although the latter solution was not complete and was completed subsequently in \cite{DiD}): %in \cite[\S 4.4.2, Theorem]{DPS} only the stronger condition $\frac{c_n^\uu(x)}{b_n^\uu}\to 0$ was considered, missing in this way the elements $\bar x\in t_\uu(\T)$ with $\varphi\left(\frac{c_n^\uu(x)}{b_n^\uu}\right)\to 0$, but $\frac{c_n^\uu(x)}{b_n^\uu}\not \to 0$). %This gap was filled in \cite{DiS:D1}. 

%To this end all these authors used the fact that for every $x\in [0,1)$ there exists a unique sequence $(c_n)_{n\in\N_+}$ in $\N$ such that 
%\begin{equation}\label{firsrEq}
%{x= \sum_{n=1}^\infty\frac{c_n}{(n+1)!},}
%\end{equation}
%with $c_n<n+1$ for every $n\in\N_+$ and $c_n<n$ for infinitely many $n\in\N_+$. A more general property is described in Theorem \ref{Peron}. 

\begin{theorem}[{\cite{DPS,Bo2,DiD}}]\label{Prob:A}
Let $x\in[0,1)$ and $\mathbf n=((n+1)!)$. Then $\bar x\in t_\mathbf n(\T)$ if and only if $\varphi\left(\frac{c_n}{n+1}\right)\to 0$.
\end{theorem}

In this paper we are mainly interested in the following particular cases of this theorem, namely, when $\supp(x)$ is either $b$-bounded or $b$-divergent.

\begin{corollary}[{\cite[Corollary 2.4]{DiD},~\cite[Corollary 3.2]{DI}}] \label{bounded}
Let $\uu\in\A$ and $x\in [0,1)$ with $S=\supp(x)$. If $S$ is $b$-bounded, then %the following conditions are equivalent: 
 $\bar x \in t_{\uu}(\T)$ if and only if $S$ is finite, namely, $\bar x$ is torsion.
%\begin{enumerate}[(1)]
%  \item $\bar x \in t_{\uu}(\T)$;
%  \item $S$ is finite;
%  \item $\bar x $ is torsion.
%\end{enumerate}
\end{corollary}

\begin{corollary}[{\cite[Corollary~3.4]{DI}}]\label{unbounded} 
Let $\uu\in\A$ and $x\in [0,1)$ with $S=\supp(x)$. If $S$ is $b$-divergent, then $\bar x \in t_{\uu}(\T)$ if and only if:
\begin{itemize}
    \item[$\mathrm{(I)}$] $\lim\limits_{n\in S}\varphi\left(\frac{c_{n}}{b_{n}}\right)=0$; 
    \item[$\mathrm{(II)}$] for every infinite $D\subseteq S$ such that $D-1$ is $b$-bounded, $\lim\limits_{n\in D}\frac{c_{n}}{b_{n}}=0$.
\end{itemize}
\end{corollary}

In the last few years some new trends appeared in the field of characterized subgroups, based on weaker notions of convergence. Next we recall what is necessary for this paper,  for more detail see the survey paper~\cite{BCsurvey}. For $A\subseteq\N$, denote 
$$A(n):=\{i\in A: i\leq n\}=A\cap[0,n].$$ 
Following~\cite{BDP} (see~\cite{BDH}), fixed $\alpha\in(0,1]$, the \emph{upper natural density of order $\alpha$} is $d_\alpha(A):=\limsup_{n\to\infty}\frac{|A(n)|}{n^\alpha}.$
For $\alpha=1$, $d:=d_1$ is the classical upper natural density. As in~\cite{BDH}, a sequence $(x_n)$ in $\T$ is said to \emph{$\alpha$-statistically converge} to $x\in \T$, denoted by $x_n\overset{s_\alpha}\longrightarrow x$, if for every neighborhood $U$ of $x$ in $\T$, $d_\alpha(\{n\in\N: x_n\not\in U\})=0$.

Following~\cite{BDH}, for a sequence of integers $\uu$ and $\alpha\in(0,1]$, let 
$t_\uu^\alpha(\T):=\{x\in\T: u_nx\overset{s_\alpha}\longrightarrow 0\}$. A subgroup $H$ of $\T$ is \emph{$\alpha$-statistically characterized} if there exists a sequence of integers $\uu$ such that $H=t_\uu^\alpha(\T)$, and the elements of $t_\uu^\alpha(\T)$ are called \emph{topologically $\uu_\alpha$-torsion elements of $\T$}. 
The $1$-statistically characterized subgroups of $\T$ are exactly the \emph{statistically characterized} $t_\uu^s(\T)$ ones introduced and studied in the seminal paper~\cite{DDB}.

\smallskip
A complete description of the topologically $\uu_\alpha$-torsion elements of $\T$ for $\uu\in\A$ and any $\alpha\in(0,1]$ was given in~\cite[Theorem 2.1]{DG1}, which is the counterpart of Theorem~\ref{DiD}. Moreover, inspired by Corollaries~\ref{bounded} and~\ref{unbounded}, respectively, the following consequences of~\cite[Theorem 2.1]{DG1} were presented in~\cite{DG1}. 
For two subsets $A, B$ of $\N$, the notation $A \subseteq^{\alpha } B$ means that $d_{\alpha} (A \setminus B)=0$ and $A\subseteq_\alpha B$ that $A\subseteq B$ and $B\subseteq^\alpha A$.

The following corollary, concerning the case of $b$-bounded support, is~\cite[Corollary~3.9]{DG1} but a careful look at its proof in~\cite{DG1} shows that only the necessity is proved (twice), while the other implication (i.e., the sufficiency) is not discussed there at all. 
As a byproduct of our results (more specifically, Corollary~\ref{CoroGhosh} -- see also Remark \ref{New:Rem}) we fix the problem, so Corollary~\ref{DG13.9} holds true.

\begin{corollary}[{\cite[Corollary 3.9]{DG1}}]\label{DG13.9}
Let $\uu\in\mathcal A$, $\alpha\in(0,1]$ and $x \in [0,1)$ with $S=\supp(x)$ and $S_b=\supp_b(x)$.
If $S$ is $b$-bounded, then $\bar x\in t_\uu^{\alpha}(\T)$ if and only if: 
\begin{itemize}
   \item[$\mathrm{(i)}$] $S+1\subseteq^\alpha S$; 
   \item[$\mathrm{(ii)}$] $S\subseteq^\alpha S_b$.
\end{itemize}
\end{corollary}

The case of $b$-divergent support is handled by the following result, in which we slightly modify condition $\mathrm{(I)}$, as the set $D\subseteq_\alpha S$ with  $\lim\limits_{n \in D} \varphi\left(\frac{c_n}{b_n}\right)=0$ need not exist in case $d_\alpha(D)=0$ (see Remark~\ref{notIx}).

\begin{corollary}[{\cite[Corollary 3.10]{DG1}}]\label{C3.10}
Let $\uu\in\mathcal A$, $\alpha\in(0,1]$ and $x \in [0,1)$ with $S=\supp(x)$. If $S$ is $b$-divergent, then $\bar x\in t_\uu^{\alpha}(\T)$ if and only if: 
\begin{itemize}
\item[$\mathrm{(I)}$] either $d_\alpha(S)=0$ or $d_\alpha(S)>0$ and there exists $D\subseteq_\alpha S$ such that $\lim\limits_{n \in D} \varphi\left(\frac{c_n}{b_n}\right)=0$; 
\item[$\mathrm{(II)}$] for every $D\in\P(\N)$ such that $D \subseteq^\alpha S$ and $D-1$ is $b$-bounded, there exists $D' \subseteq_\alpha D$ such that $\lim\limits_{n \in D'} \frac{c_n}{b_n}=0$.
\end{itemize} 
\end{corollary}

\subsection{Ideal convergence and main results}

The $\alpha$-statistical convergence and the usual one are particular cases of the ideal convergence, introduced by Cartan~\cite{Cartan} as follows.

\smallskip
First recall that a non-empty subfamily $\I$ of the power set $\P(\N)$ of $\N$ is an \emph{ideal of $\N$} if 
%\begin{itemize}
%\item[-] 
for every $A,B\in\I$,  $A\cup B\in \I$,
%\item[-] 
and for $A, B \in \P(\N)$, $B\in\I$ implies $A\in\I$ whenever $A\subseteq B$. 
%\end{itemize}
 An easy example of ideal of $\N$ is the family $\F in$ of all finite subsets of $\N$, and we always assume that an ideal $\I$ of $\N$ contains $\F in$, that is, $\I$ is \emph{free}. 
An ideal $\I$ of $\N$ is a \emph{$P$-ideal}  if for every countable family $\{A_n:n\in\N\}\subseteq \I$ there exists $A\subseteq \N$ such that $A_n\subseteq^* A$ for all $n\in\N$. Moreover, following~\cite[Definition 2.4]{DasGhosh}, the ideal $\mathcal I$ is \emph{translation invariant} if for every $A\in\mathcal \I$, $(A+n)\cap\N\in\mathcal I$ for every $n\in\Z$.
%For undefined terms see \S\ref{NT}, and take into account that 
The ideals $\F in$ and $\I_\alpha$, for every $\alpha\in(0,1]$, are translation invariant (analytic) free $P$-ideals.

For two subsets $A, B$ of $\N$ and an ideal $\I$ of $\N$, we use the following notations:
\[\begin{split}A \subseteq^{\I} B\ \text{if}\ A \setminus B \in \I,\quad A \subseteq_{\I} B\ \text{if}\ A \subseteq B\ \text{and}\ B \setminus A \in \I \\
\quad\text{and}\quad 
A =^{\I} B\ \text{if}\ A \subseteq^{\I} B\ \text{and}\ B \subseteq^{\I} A\quad \text{(i.e., $A\Delta B\in\I$)}.\end{split}\]
Fixed $\uu\in\A$ and a free ideal $\I$ of $\N$, we say that a subset $X$ of $\N$ is \emph{$b$-bounded mod $\I$} (resp., \emph{$b$-divergent mod $\I$}) if there exists a $b$-bounded (resp., $b$-divergent) $A\subseteq_\I X$.

\smallskip
For a free ideal $\I$ of $\N$, a sequence $(x_n)$ in $\T$ is said to {\em $\I$-converge} to a point $x\in \T$, denoted by $x_n\overset{\I}\longrightarrow x$, if $\{n\in \N: x_n \not \in U\}\in \I$ for every neighborhood $U$ of $x$ in $\T$.  Following~\cite{DasGhosh,Ghosh}, given a sequence $\uu$ in $\Z$, let $t_\uu^\I(\T) :=\{x\in \T: u_nx\overset{\I}\longrightarrow 0 \}$. A subgroup $H$ of $\T$ is {\em $\I$-characterized} if there exists a sequence of integers $\uu$ such that $H= t_\uu^\I(\T)$, and the elements of $t_\uu^\I(\T)$ are called \emph{topologically $\uu_\I$-torsion elements of $\T$}. 

The usual convergence is the $\F in$-convergence, while, for every $\alpha\in(0,1]$, the $\alpha$-statistical convergence is the $\I_\alpha$-convergence, where $\I_\alpha:=\{A\subseteq\N: d_\alpha(A)=0\}$; usually one denotes $\I_d:=\I_1$. So, first of all, $t_\uu^{\F in}(\T)=t_\uu(\T)$ and the inclusion $t_\uu(\T)\subseteq t_\uu^\I(\T)$ always holds, as $\I$ is free.
Moreover, $t_\uu^{\I_\alpha}(\T)=t_\uu^\alpha(\T)$ for every $\alpha\in(0,1]$, and in particular $t_\uu^{\I_d}(\T)=t_\uu^s(\T)$.

\smallskip
Next comes the main theorem of~\cite{I-torsion}, which covers~\cite[Theorem~2.3]{DI} (se Theorem~\ref{DiD}) and~\cite[Theorem 2.1]{DG1} mentioned before.  

\begin{theorem}[\cite{I-torsion}]\label{conjecture} 
Let $\uu\in\mathcal A$, let $\I$ be a translation invariant free $P$-ideal of $\N$ and $x \in[0,1)$ with $S=\supp(x)$ and $S_b=\supp_b(x)$. Then $\bar x\in t^{\I}_\uu(\T)$ if and only if for all $A \in\P(\N)\setminus\I$ the following holds.
 \begin{itemize}
   \item[$\ax$]  If $A$ is $b$-bounded, then:
 \begin{itemize}
     \item[$\axuno$] if $A \subseteq^{\I} S$, then $A+1\subseteq^{\I} S$,  $A \subseteq^{\I} S_b$ and there exists $A' \subseteq_{\I} A$ such that  $\lim\limits_{n \in A'} \frac{c_{n+1}+1}{b_{n+1}}=1$;
     \item[$\axdue$]  if $A\cap S \in \I$, then there exists $A'\subseteq_\I A$ such that 
$\lim\limits_{n \in A'} \frac{c_{n+1}}{b_{n+1}}=0$.
\end{itemize}
\item[$\bx$] If $A$ is $b$-divergent, then there exists $B \subseteq_{\I} A$ such that $\lim\limits_{n \in B} \varphi\left(\frac{c_n}{b_n}\right)=0$.
\end{itemize}
\end{theorem} 

A version of this theorem was given in~\cite[Theorem~2.9]{Ghosh} under the additional hypothesis that the ideal has to be analytic and with additional conditions also in $\axuno$ and $\axdue$. Unfortunately, that theorem presents a gap in its proof (similar to the gap in the proof of~\cite[Theorem~2.1]{DG1}) which is described in detail in~\cite{I-torsion}, where we also show that~\cite[Theorem~2.9]{Ghosh} can be deduced from Theorem~\ref{conjecture}.

\medskip
One of the main result of this paper is the following characterization of the topologically $\uu_\I$-torsion elements of $\T$ with $b$-bounded support (see \S\ref{suppbbounded} for a proof). 
For $\uu\in\A$, $\I$ a translation invariant free ideal of $\N$ and $x\in[0,1)$ with $S=\supp(x)$ and $S_b=\supp_b(x)$, let
\begin{itemize}
\item[$\ix$] $S+1\subseteq^\I S$;
\item[$\iix$] $S_b\subseteq_\I S$.
\end{itemize}
Moreover, we briefly write $\Ax$ for the conjunction $\ix\&\iix\&\axdue$. 

\begin{theorem}\label{Nuovo:Th}
Let $\uu\in\A$, let $\I$ be a translation invariant free $P$-ideal of $\N$ and $x\in[0,1)$. If $\supp(x)$ is $b$-bounded mod $\I$, then $\bar x\in t_\uu^\I(\T)$ if and only if $\Ax$ holds.
\end{theorem}

%\begin{remark}
From Theorem~\ref{Nuovo:Th} one can deduce Corollary~\ref{bounded} as follows, with $\I = \F in$. Let $\uu\in\A$ and $x\in[0,1)$ with $S=\supp(x)$ $b$-bounded and $S_b=\supp_b(x)$. 
It is clear that if $S$ is finite, then $\bar x\in t_\uu(\T) = t_\uu^\I(\T)$ as $\uu$ is arithmetic. 
Assume that $\bar x\in t_\uu(\T)$. By Theorem~\ref{Nuovo:Th}, $\Ax$ holds. Assume for a contradiction that $S$ is infinite. By $\ix$, $(S+1)\setminus S$ is finite and so $S$ is cofinite by~\cite[Claim~3.1]{DI}. By $\iix$, $S_b$ is cofinite, too, which is a contradiction.
%\end{remark}

\smallskip
In~\cite[Corollary 2.12]{Ghosh}, the following characterization was given of the topologically $\uu_\I$-torsion elements with $b$-bounded support: {\em if $\uu\in\A$ and $\I$ is a translation invariant analytic free $P$-ideal, then  for $x\in[0,1)$ with $S=\supp(x)$ $b$-bounded and $S_b=\supp_b(x)$, $\bar x \in t^{\I}_\uu(\T)$ is equivalent to $\ix\&\iix$.} (So the condition $\axdue$ is missing with respect to Theorem~\ref{Nuovo:Th}.)
%\begin{itemize}
%\item[$\ix$] $S+1\subseteq^\I S$;
%\item[$\iix$] $S_b\subseteq_\I S$.
%\end{itemize} }
This corollary is left without proof in~\cite{Ghosh}, where the reader finds only this comment: ``The proof follows from similar line of arguments as in~\cite[Corollary~3.9]{DG1} and so is omitted." As commented at the end of \S \ref{Sec:Hist}, the sufficiency in~\cite[Corollary~3.9]{DG1} (see Corollary~\ref{DG13.9}) is not proved in that paper. 
Hence, also in~\cite[Corollary 2.12]{Ghosh} the sufficiency remains unproved (while the necessity, i.e., $\bar x\in t_\uu^\I(\T)\Rightarrow \ix\&\iix$ when $x\in [0,1)$ has $b$-bounded support, can be deduced  from Theorem~\ref{conjecture}, as we show in Lemma~\ref{Necessity2.12*}).

\smallskip
In order to face the general case (when $\supp(x)$ need not be $b$-bounded mod $\I$), we add a third condition to $\ix$ and $\iix$; namely: 
\begin{itemize}
 \item[$\iiix$] $S-1\subseteq^\I S$.
\end{itemize} 
and we abbreviate $\Tx=\ix\&\iix\&\iiix$.
In general, $\ix\&\iix\Rightarrow\axuno$ and $\iiix\Rightarrow\axdue$, so in particular $\Tx\Rightarrow\Ax\Rightarrow\ax$ (see Proposition~\ref{primo} and the diagram below).
Furthermore, $\Tx$ ensures $\bar x\in t_\uu^\I(\T)$ regardless of whether $\supp(x)$ is $b$-bounded or not (see Theorem~\ref{in}).
As a consequence, we get that~\cite[Corollary 2.12]{Ghosh} is true under the additional hypothesis on the ideal to be {\em nested} (see Definition~\ref{nesteddef}), 
as $\ix\&\iix\Rightarrow\iiix$ (namely, $\ix\&\iix\equiv\Tx$) under this additional hypothesis.
%but relaxing the analyticity (see Corollary~\ref{CoroGhosh}). 
Since $\I_\alpha$ is nested for every $\alpha\in(0,1]$, as a by-product we deduce that~\cite[Corollary~3.9]{DG1} (see Corollary~\ref{DG13.9}) holds true.

The crucial notion of a nested ideal %(given only in Definition~\ref{nesteddef}) 
was motivated by the helpful combinatorial notion of (left and right) boundary of an infinite subset of $\N$ that we anticipate in \S\ref{Sec:boundary}.
The translation invariant free ideals that are nested can be characterized as follows in terms of topologically $\uu_\I$-torsion elements of $\T$:

\begin{theorem}\label{Cor:May12prec}
The following conditions are equivalent for a translation invariant free ideal $\I$ of $\N$: 
\begin{itemize}
\item[(1)] $\I$ is nested; 
\item[(2)] for every $\uu\in\mathcal A$ the following implications hold true: 
\begin{equation}\label{Eq:May23}
\text{for every $x\in[0,1)$, $\ix\&\iix \Rightarrow\ \bar{x}\in t_\uu^\I(\T)$;}
\end{equation}
\item[(3)]~\eqref{Eq:May23} holds for every $b$-bounded $\uu\in\mathcal A$; 
\item[(4)]~\eqref{Eq:May23}  holds for $\uu =(2^n) \in\mathcal A$;
\item[(5)]  there exists a $b$-bounded $\uu \in \A$ such that~\eqref{Eq:May23} holds. 
\end{itemize}
\end{theorem}

In view of Theorem~\ref{Cor:May12prec}, the family of examples of translation invariant analytic free $P$-ideals that are not nested contained in \S\ref{Countersec}, is a family of counterexamples to~\cite[Corollary 2.12]{Ghosh}

\smallskip
In the following diagram we collect the implications that we discuss in the paper; let $\I$ be a translation invariant free $P$-ideal of $\N$, $\uu\in\A$ and $x\in[0,1)$ with $S=\supp(x)$. Some of the implications hold true in general, while others only under additional hypotheses that we attach to the relevant arrows.  
In the latter case one can always replace the condition ``$b$-bounded" by ``$b$-bounded mod $\I$".

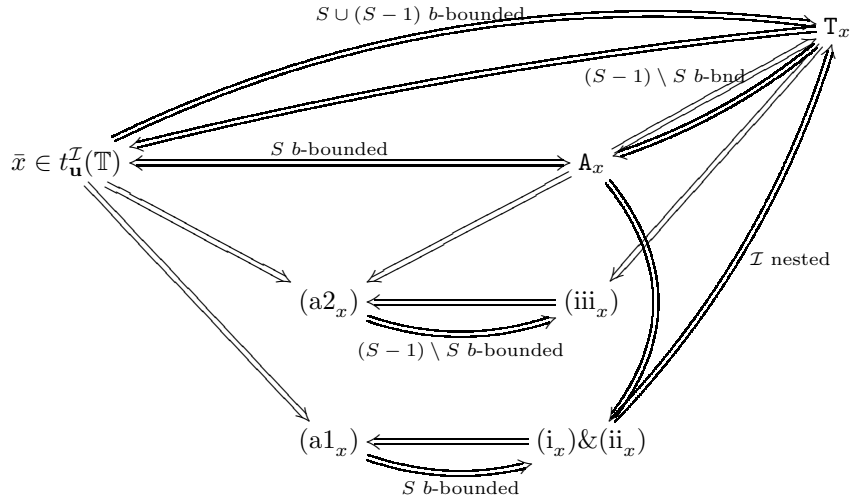
$$\xymatrix@R3pc@C5pc{
 & & &  \Tx \ar@{=>}@/^0.5pc/[ld]\ar@{=>}[ldd] \ar@{=>}@/_0.5pc/[llld] \\
 \bar x  \in t_\uu^\I(\T) \ar@{=>}[rd]\ar@{=>}[rdd] \ar@{=>}@/^2pc/[rrru]^{\text{$S\cup(S-1)$ $b$-bounded}}   & & \Ax \ar@{=>}[ur]^{\text{$(S-1)\setminus S$ $b$-bnd}}\ar@{<=>}[ll]_{\text{$S$ $b$-bounded}} \ar@{=>}[dl]\ar@{=>}@/^2pc/[dd] \\
  & \axdue   \ar@{=>}@/_1pc/[r]_{\text{$(S-1)\setminus S$ $b$-bounded}} &\iiix  \ar@{=>}[l] &\\
  & \axuno \ar@{=>}@/_1pc/[r]_{\text{$S$ $b$-bounded}}  &\ix\&\iix \ar@{=>}[l] \ar@{=>}@/_1pc/[uuur]_{\text{$\I$ nested}} & 
}$$

All the implications in the diagram that hold under an additional assumption do not hold true in general. In fact, always $\Tx\Rightarrow\bar{x}\in t_\uu^\I(\T)$ by Theorem~\ref{in}, while Corollary~\ref{b-bounded} implies that $\Tx \Leftrightarrow \bar x\in t_\uu^\I(\T)$ when $\uu$ is $b$-bounded.  Nevertheless, the necessity of $\Tx$ for $\bar x\in t_\uu^\I(\T)$ does not hold true in general (see Example~\ref{ce}); more precisely, $\bar x\in t_\uu^\I(\T)$ does not imply any of $\ix$, $\iix$ and $\iiix$ (in particular, $\axuno\Rightarrow\ix\&\iix$ and $\axdue\Rightarrow\iiix$ fail in general). So, also the necessity of $\Ax$ does not hold true in general under the assumption that $\bar x\in t_\uu^\I(\T)$. 

As there exist non-nested translation invariant free ideals of $\N$ (see \S\ref{Countersec}), also the implication $\ix\&\iix\Rightarrow\Tx$ does not hold in general (see Proposition~\ref{notnested->}). Example~\ref{Exa:2:osserv} witnesses the failure of the implication $\Ax \Rightarrow \Tx$.

\medskip 
In \S\ref{Sec:Coro3.10:DG1} we analyze the case when $\supp(x)$ is $b$-divergent ending up with the following result that extends Corollaries~\ref{unbounded} and~\ref{C3.10}. We directly prove the sufficiency of the conditions $\Ix$ and $\IIx$ for $\bar x\in t_\uu^\I(\T)$, without making recourse to the sufficiency of the conditions in Theorem~\ref{conjecture}, the proof of which requires much more effort.

\begin{theorem}\label{Last:corollary} 
Let $\uu\in\A$, let $\I$ be a translation invariant free $P$-ideal of $\N$ and $x\in[0,1)$ with $S=\supp(x)$ $b$-divergent  mod $\I$. Then $\bar x\in t_\uu^\I(\T)$ if and only if: 
\begin{itemize}
  \item[$\Ix$]  either $S\in\I$ or if $S\not\in\I$, then there exists $D\subseteq_\I S$ such that $\lim\limits_{n\in D}\varphi\left(\frac{c_n}{b_n}\right)=0$;
  \item[$\IIx$] for every $D\in\P(\N)\setminus\I$ with $D\subseteq^\I S$ and  $D-1$ $b$-bounded, there exists $D'\subseteq_\I D$ such that $\lim\limits_{n\in D'}\frac{c_n}{b_n}=0$.
\end{itemize}
\end{theorem}

Since $S=\supp(x)$ is $b$-divergent  mod $\I$,  the second part of $\Ix$ is equivalent to: there exists $D\subseteq_\I\N$ with $\lim\limits_{n\in D}\varphi\left(\frac{c_n}{b_n}\right)=0$.

\smallskip
The final \S\ref{Sec:I-spl} is focused on a result that simultaneously generalizes Theorem \ref{Last:corollary} and Theorem \ref{Nuovo:Th}.
More precisely, inspired by \cite[Definition 3.1]{Ghosh} (which was, in turn, inspired by the splitting property of a sequence from~\cite{DG1,DI}), given $\uu\in\A$ and a free ideal $\I$ we say that a subset $A$ of $\N$ has the \emph{$\I$-splitting property} if it has the form $A=B\sqcup D$, with $B$ $b$-bounded mod $\I$ and $D$ $b$-bounded mod $\I$. In Theorem \ref{Lemma1}, for $\uu\in \A$ and a translation invariant free $P$-ideal $\I$ of $\N$, we describe the elements $\bar x$ of the subgroup $t_\uu^\I(\T)$ such that $\supp(x)$ has the $\I$-splitting property.

\subsection{Notation and terminology}\label{NT}

As usual, denote by $\N$ the set of natural numbers and by $\P(\N)$ its power set. For a subset $A$ of $\N$, we denote $A^*=\N\setminus A$, and for an ideal $\I$ of $\N$, $\I^*=\{A^*:A\in\I\}$.

Moreover, let $\N_+=\{n\in\N:n>0\}$, while $\Z$, $\Q$ and $\R$ are respectively the group of integers, rationals and reals. 
We denote as usual $[a,b]$ and $(a,b)$ a closed and an open interval of $\R$. Moreover, let $[a,b]_\N=[a,b]\cap\N$ and $(a,b)_\N=(a,b)\cap\N$.
For $x\in\R$, we denote by $\{x\}$ its fractional part and by $\lfloor x\rfloor$ its integer part; moreover, for $x,y\in\R$, we write $x\equiv_\Z y$ to denote that $x-y\in\Z$. 
Let $\T=\R/\Z$ and let $\| -\|\colon\T\to[0,1/2]$ be the norm on $\T$ defined by $\|x+\Z\|=\min\{|x-n|:n\in\N\}$ for every $x\in\R$. 
For $m\in \N_+$, $\Z(m) \cong \Z/m\Z$ denotes the cyclic group of order $m$, and for $p$ a prime let $\Z(p^\infty)$ denote the Pr\"ufer group.

\smallskip
Recall that a subset $Y$ of a metric space $X$ is \emph{analytic} if it is a continuous image of a Borel set. So, we call an ideal $\I$ of $\N$ \emph{analytic} if it is an analytic subset of the Boolean ring $(\P(\N),\bigtriangleup,\cap)$ identified with $\Z(2)^\N$ and endowed with the product topology of the discrete topology of $\Z(2)$.
 
%The ideals $\F in$ and $\I_\alpha$ are translation invariant analytic free $P$-ideals.

As proved in~\cite{BCsurvey}, every translation invariant free $P$-ideal of $\N$ is \emph{unboundedly convex}, namely, $\I$ contains an infinite set $A=\bigcup_{n\in\N}[a_n,b_n]_\N$, where $a_{n+1}>b_n+1$, such that $\sup_{n\in\N}(b_n-a_n+1)=\infty$; roughly, $A$ has arbitrarily large intervals. The term {unboundedly convex} was coined in~\cite{BCpaper}, while in~\cite{BCsurvey} this is equivalent to the negation of the property called $\stellaa$ there.
 
\smallskip
 Fixed $\uu\in\A$, we use the following easy estimations. 
First, \begin{equation}\label{finiteestimation}
\sum_{i=l}^{r} \frac{b_i-1}{u_i}= \frac{1}{u_{l-1}}-\frac{1}{u_{r}}\quad\text{for every $l,r\in\N_+$, with $l\leq r$,}
\end{equation}
as $\sum_{i=l}^{r} \frac{b_i-1}{u_i}=\sum_{i=l}^{r}\left( \frac{1}{u_{i-1}}-\frac{1}{u_i}\right)= \frac{1}{u_{l-1}}-\frac{1}{u_{r}}$.
Similarly, 
\begin{equation}\label{NEW:Lemma}
\sum_{n=k+1}^\infty\frac{c_n}{u_n}\leq \frac{1}{u_k}\quad\text{for every $k\in\N$},
\end{equation}
as $\sum_{n=k+1}^\infty\frac{c_n}{u_n}\leq \sum_{n=k+1}^\infty\frac{b_n-1}{u_n}=\sum_{n=k}^\infty\left(\frac{1}{u_n}-\frac{1}{u_{n+1}}\right)=\frac{1}{u_k}.$

\section{General relations}

In the next lemma we see that $\axdue$ is equivalent to a simpler condition, that we use in the paper without giving always reference to the lemma.

\begin{lemma}\label{ax2p}
Let $\uu\in\A$, let $\I$ be a free ideal of $\N$ and $x\in[0,1)$ with $S=\supp(x)$. Then $\axdue$ holds if and only if:
\begin{itemize}
\item[$\dageq$] for every $A\in\P(\N)\setminus\I$ $b$-bounded, if $A\cap S \in \I$, then $(A+1)\cap S\in \I$.
\end{itemize}
\end{lemma}
\begin{proof}
To prove that $\dageq\Rightarrow\axdue$, assume that $A\in\P(\N)\setminus\I$ is $b$-bounded and $A\cap S \in \I$. Then by $\dageq$, $(A+1)\cap S\in \I$, namely, $A\cap (S-1)\in\I$.
Let $A'=A\setminus(S-1)\subseteq_\I A$.
Then $c_{n+1}=0$ for every $n\in A'$, and so $\lim\limits_{n\in A'}\frac{c_{n+1}}{b_{n+1}}=0$. 
Hence, $\axdue$ holds.

To prove the converse implication $\axdue\Rightarrow\dageq$, let $A\in\P(\N)\setminus\I$ be $b$-bounded with $A \cap S \in \I$. Then $\axdue$ provides a subset $A'\subseteq_\I A$ such that $\lim\limits_{n\in A'}\frac{c_{n+1}}{b_{n+1}}=0$. We prove that $A'':= A' \cap (S-1)$ is finite. 
Indeed, assume by contradiction that $A''$ is infinite. Then $\lim\limits_{n\in A''}\frac{c_{n+1}}{b_{n+1}}=0 $, since $A''  \subseteq A'$.
As $A'' \subseteq S -1$, $c_{n+1} \ne0$ for every $n \in A''$. This gives 
$$\lim\limits_{n\in A''}\frac{1}{b_{n+1}}\leq \lim\limits_{n\in A''}\frac{c_{n+1}}{b_{n+1}}=0. $$ 
Thus $\lim\limits_{n\in A''} b_{n+1}=\infty$, a contradiction to the fact that $A''$ is $b$-bounded, as a subset of $A$. This establishes the finiteness of $A''$. 
Since, $(A'+1)\cap S=A''+1$ is finite, and so in $\I$, and since moreover $A'+1\subseteq_\I A+1$, we conclude that $(A+1)\cap S\in\I$, namely, $\dageq$ holds.
\end{proof}

\begin{remark}\label{ax1p}
Let $\uu\in\A$, let $\I$ be a free ideal of $\N$ and $x\in[0,1)$ with $S=\supp(x)$.
Analogously to $\dageq$, one can introduce the following property related to $\axuno$:
\begin{itemize}
\item[$\stareq$] for every $A\in\P(\N)\setminus\I$ $b$-bounded, if $A \subseteq^{\I} S$, then $A+1\subseteq^{\I} S$, $A \subseteq^{\I} S_b$ and $A \cap (S_b-1) \subseteq_\I A$.
\end{itemize}

\smallskip
(a) First we verify that $\stareq\Rightarrow \axuno$.
Assume that $A\in\P(\N)\setminus\I$ is $b$-bounded and $A\subseteq^{\I} S$. Then by $\stareq$, $A \subseteq^{\I} S_b$, $A+1\subseteq^{\I} S$ and $A':=A\cap (S_b-1)\subseteq_\I A$.
Since $A' + 1 \subseteq S_b$, it follows that $c_{n+1}=b_{n+1}-1$ for every $n \in A'$, and so $\lim\limits_{n\in A'}\frac{c_{n+1}+1}{b_{n+1}}=1$. 

\smallskip
(b) Now we see that for $\I=\F in$, $\bar x\in t_\uu(\T)\not\Rightarrow\stareq$, so in particular $\stareq$ is strictly stronger than $\axuno$. 
In order to define the sequence $\uu\in\A$, we fix a partition of $\N_+ = B\sqcup D$, with $D=2\N_+$ and $B=2\N+1$, so that $B + 1= D$. The sequence $\uu\in\A$ is completely determined 
by the ratios $b_n$, so let $b_n=2$ for $n \in B$ and $b_n=n$ for $n \in D$. Then $D$ is $b$-divergent and $B$ is $b$-bounded. 

Let $x=\sum_{n\in 2\N_+}\frac{b_n-2}{u_n} + \sum_{n\in 2\N_+ +1}\frac{b_n-1}{u_n}$. Then  $\supp(x) = \N_+\setminus \{1\}$ and $\supp_b(x) = B \setminus \{1\}$. 
Let $A$ be an infinite $b$-bounded subset of $\N$. Then $A\subseteq^* B$, and so also $A\subseteq^* \supp_b(x)$. Moreover, $A+1\subseteq^* D$ and $\lim\limits_{n\in A+1}\frac{c_{n}+1}{b_n}=\lim\limits_{n\in A+1}\frac{b_n-1}{b_n}=\lim\limits_{n\in A+1}\frac{n-1}{n}=1$. This shows that $\axuno$ holds.
On the other hand, $(A+1)\cap \supp_b(x)$ is finite, so $\stareq$ does not hold.

Now $\axdue$ is trivially satisfied and it easy to see that also $\bx$ holds, as if $E$ is a $b$-divergent subset of $\N$, then $E\subseteq^*D$ and $\lim\limits_{n\in S}\frac{c_n}{b_n}=\lim\limits_{n\in D}\frac{n-2}{n}=1$. Hence, $\bar x\in t_\uu(\T)$ by Theorem~\ref{conjecture}.
\end{remark}
 
Next, we clarify the impact of $\Tx$ and $\Ax$ on $\ax$: 

\begin{proposition}\label{primo}{\rm [$\Tx\Rightarrow \Ax \Rightarrow \ax$]} 
Let $\uu\in\A$, let $\I$ be a free ideal of $\N$ and $x\in[0,1)$. Then:
$$\ix\&\iix\Rightarrow \axuno\quad \text{and}\quad \iiix \Rightarrow \axdue.$$
Consequently, $\Tx \Rightarrow \Ax\Rightarrow \ax$.
\end{proposition}
\begin{proof} 
Let $S=\supp(x)$ and $S_b=\supp_b(x)$.
%\smallskip 
Assume $\ix\&\iix$ and that $A\in\P(\N)\setminus\I$ is $b$-bounded with $A\subseteq^\I S$. By hypotheses, $S+1 \subseteq^{\I} S \subseteq^{\I} S_b$. This gives also $S+1  \subseteq^{\I} S_b$ and $S-1 \subseteq^{\I} S_b-1$.
As $A \subseteq^{\I} S$, then $A \subseteq^{\I} S_b$ and $A+1 \subseteq^{\I} S+1$, so $A+1 \subseteq^{\I} S$ by our assumption $S+1  \subseteq^{\I} S$. From $A+1 \subseteq^{\I} S$ and $S-1 \subseteq^{\I} S_b-1$, and  taking into account that $\I$ is translation invariant,  we deduce that $A\subseteq^\I S-1 \subseteq^\I S_b-1$, so $A \subseteq^\I S_b-1$.  Hence, $A\cap(S_b-1)\subseteq_\I A$, namely, $\stareq$ holds, which implies $\axuno$ by Remark~\ref{ax1p}.

\smallskip 
Now atssume  that $\iiix$ holds and that $A\in\P(\N)\setminus\I$ is $b$-bounded with $A\cap S\in\I$. Then also $A\cap (S-1)\in\I$, since $S-1\subseteq^\I S$ by $\iiix$.
Therefore, $(A+1)\cap S=(A\cap (S-1))+1\in\I$ as well, so $\dageq$ holds, which is equivalent to $\axdue$ by Lemma~\ref{ax2p}.
\end{proof}

The last statement of Proposition~\ref{primo} follows from Theorem~\ref{in} and Theorem~\ref{conjecture}. Nevertheless, we prefer to give a direct proof of the more precise version, relating $\ix\&\iix$ and $\axuno$ (resp., $\iiix$ and $\axdue$). On the other hand, the new condition $\iiix$ can be strictly stronger than $\axdue$, as witnessed by Example~\ref{ce}.

\section{Sufficient conditions for $\bar x\in t_\uu^\I(\T)$}
 
We start this section by recalling that $\supp(x)\in\I$ is a basic sufficient condition to get $\bar x\in t_\uu^\I(\T)$:

\begin{lemma}[{\cite[Lemma 2.2]{Ghosh}}]\label{Lemma2.2}
Let $\uu\in\A$, let $\I$ be a translation invariant proper free ideal of $\N$ and $x\in[0,1)$. If $\supp(x)\in\I$, then $\bar x\in t_\uu^\I(\T)$.
\end{lemma}

\subsection{Two boundary operators in $\P(\N)$}\label{Sec:boundary}

There is a quite natural way to define the left and the right boundary of a subset $S$ of $\N$. In case $S = [a,b]_\N$ is an interval, we simply put 
$$\lambda(S) := \{a\} \quad \text{and}\quad \rho(S):= \{b\}.$$
%namely, the left and the right, respectively, end of the boundary of $S$. 
Obviously, $\lambda(S) = S \setminus (S+1)$ and $\rho(S) = S \setminus (S-1)$ when $|S|>1$ is not degenerate. This remains true also when $S= \{a\}$ is singleton (degenerate), as $\lambda(S) = \rho(S) =S$, while $S \setminus (S+1) = S \setminus (S-1) = S$ in this case.
It can be helpful to introduce also the set $i(S)$ of all {\em isolated points} of $S$, namely,
$$i(S) = \{s\in S: s\not \in (S-1) \cup (S+1)\} = S \setminus ((S-1) \cup (S+1)) = (S \setminus (S-1)) \cap (S \setminus (S+1)).$$

In the general case we simply use the equalities $\lambda(S) = S \setminus (S+1)$ and $\rho(S) = S \setminus (S-1)$
valid for intervals $S = [a,b]_\N$, and define the left and the right boundary of an arbitrary $\emptyset \ne S \in \P(\N)$ in the same way: 
$$\lambda(S) := S \setminus (S+1)\quad\text{and}\quad\rho(S):= S \setminus (S-1).$$
Then $i(S) = \lambda(S) \cap \rho(S)$, i.e., the isolated points are exactly those points of $S$ that are both  left and right boundary.

The next useful claim and its proof show that this choice is equally natural: 

\begin{claim}\label{nesteddef*} For $\emptyset \ne S \in \P(\N)$, the following conditions are equivalent: 
\begin{itemize}
  \item[(1)] $S$ is either 
  %\maltese
   finite or cofinite; 
  \item[(2)] $\lambda(S)$ is finite; 
  \item[(3)] $\rho(S)$ is finite. 
\end{itemize}
\end{claim}
\begin{proof} 
The implication (1)$\Rightarrow$(2)\&(3) is obvious.  

Now assume that $S$ is neither finite nor cofinite. Then $S$ can be written as $S=\bigcup_{n \in \N}[l_n,r_n]_\N$ where 
\begin{equation}\label{(*)}
l_n\leq r_n<l_{n+1}-1\quad \text{for every}\ n\in \N. 
\end{equation}
Obviously, both $S \setminus (S+1) = \{l_n: n \in \N \}$ and $S \setminus (S-1) =  \{r_n: n \in \N\}$ are infinite now. This proves the remaining implications. 
\end{proof}

In the sequel we often deal with infinite non-cofinite sets $S\in \P(\N)$. 

\begin{remark}\label{rem:nested} 
One may wonder when a pair of infinite sets $L =\{l_n:n\in\N\}$, $R=\{r_n:n\in\N\}$ of $\N$ is the pair of left and right boundary of the same infinite non-cofinite set $S \in \P(\N)$. A necessary condition~\eqref{(*)}  was given above. It is easy to see that it is also sufficient (just take $S=\bigcup_{n \in \N}[l_n,r_n]_\N$). 
\end{remark}

Obviously, given $\uu\in\A$, $\I$ a free ideal of $\N$ and $x\in[0,1)$ with $S=\supp(x)$,
$$\ix\Leftrightarrow \rho(S)\in\I\quad\text{and}\quad \iiix\Leftrightarrow\lambda(S)\in\I.$$
Both conditions are significantly relaxed versions of the sufficient condition $S\in\I$ for $\bar x \in t_\uu^\I(\T)$.

\smallskip
From the necessity in Theorem~\ref{conjecture} (i.e.,~\cite[Proposition 3.3]{I-torsion}), applied to $\rho(S)$ $b$-bounded one can immediately deduce that $\bar x\in t_\uu^\I(\T)\Rightarrow \ix$:

\begin{lemma}\label{rem:nested*}{\rm [$\bar{x}\in t_\uu^\I(\T)\Rightarrow\axuno\Rightarrow\ix$]}
Let $\uu\in\A$, $\I$ a translation invariant proper free $P$-ideal of $\N$ and $x\in [0,1)$ with $S=\supp(x)$. If $\rho(S)$ is $b$-bounded, then the implication 
$\axuno\Rightarrow \rho(S)\in \I$ holds $($i.e., $\axuno\Rightarrow\ix)$; in particular, $\bar x\in t_\uu^\I(\T)\Rightarrow \rho(S)\in \I$  $($i.e., $\bar x\in t_\uu^\I(\T)\Rightarrow \ix)$ holds.
\end{lemma}
\begin{proof}
 Assume that $\rho(S)$ is $b$-bounded and that $\axuno$ holds.
If $\rho(S)\not \in \I$, applying $\axuno$ to $\rho(S)\subseteq S$ would imply $\rho(S)+1 \subseteq ^\I S$. As $\rho(S)+1$ and $S$ are disjoint, this would imply that $\rho(S)+1\in\I$ and so $\rho(S)\in \I$, a contradiction.  Hence, $\rho(S)\in\I$.
By Theorem~\ref{conjecture}, $\axuno$ holds under the assumption that $\bar x\in t_\uu^\I(\T)$.
\end{proof}

%\textbf{This should be compared with Lemma~\ref{Necessity2.12*}, where the same implication ($\bar x\in t_\uu^\I(\T)\Rightarrow \ix$) is proved under the much stronger assumption that $\supp(x)$ is $b$-bounded.}

 The following is the counterpart of Lemma~\ref{rem:nested*} for $\lambda(S)$ and $\axdue$.

\begin{lemma}\label{lambda-1} {\rm [$\axdue\Rightarrow\iiix$]}
Let $\uu\in\A$, $\I$ a translation invariant proper free $P$-ideal of $\N$ and $x\in [0,1)$ with $S=\supp(x)$. If $\lambda(S)-1$ is $b$-bounded, then the implication 
 $\axdue\Rightarrow \lambda(S)\in \I$ holds $($i.e., $\axdue\Rightarrow\iiix)$; in particular, $\bar x\in t_\uu^\I(\T)\Rightarrow \lambda(S)\in \I$  $($i.e., $\bar x\in t_\uu^\I(\T)\Rightarrow \iiix)$ holds.
\end{lemma}
\begin{proof}  Assume that $\lambda(S)-1$ is $b$-bounded and that $\axdue$ holds. If $\lambda(S)-1\in \I$, then $\lambda(S)\in\I$ since $\I$ is translation invariant.
 In case $\lambda(S)-1 \not \in \I$, we apply $\axdue$ to $\lambda(S)-1$ noting that $(\lambda(S)-1)\cap S = \emptyset \in \I$, and deduce that $\lambda(S) \cap S =\lambda(S) \in \I$, a contradiction. Hence, $\lambda(S)\in\I$. By Theorem~\ref{conjecture}, $\axdue$ holds under the assumption that $\bar x\in t_\uu^\I(\T)$.
\end{proof}

\begin{definition}\label{def:atomic}
Let $\uu\in\A$. Call an element $x\in [0,1)$ {\em atomic} if its representation has support $S = \supp(x)$ consisting of only isolated points, i.e., $\lambda(S) = \rho(S) = i(S)=S$, and all  non-zero coefficients $c_n(x)$ are equal to $1$.
\end{definition}

In the next proposition we show that Lemma~\ref{Lemma2.2} can be inverted for atomic elements of $b$-bounded support. It can be deduced from much stronger results from~\cite{I-torsion}, but to make our exposition self-contained we prefer to give here a short direct proof.
It is important to keep in mind that the implication $\supp(x)\in \I \Rightarrow \bar x\in t_\uu^\I(\T)$ holds without the assumption that $\supp(x)$ is $b$-bounded.

\begin{proposition}\label{Propo:New}
Let $\uu\in\A$, let $\I$ be a translation invariant proper free $P$-ideal of $\N$ and $x\in [0,1)$ with $S = \supp(x)$. If $x$ is atomic and $S$ is $b$-bounded, then $\bar x\in t_\uu^\I(\T)$ if and only if $S \in \I$. 
\end{proposition}
\begin{proof}
In view of Lemma~\ref{Lemma2.2} $S\in\I\Rightarrow \bar x\in t_\uu^\I(\T)$. In order to prove that 
%\maltese\footnote{Ho aggiunto "when $x \in [0,1)$ is atomic" } 
when $x \in [0,1)$ is atomic, $\bar x\in t_\uu^\I(\T)\Rightarrow S\in\I$, one can apply  Lemma~\ref{rem:nested*}, but we prefer to give a direct proof.

Write $S = \lambda(S) = \rho(S) = i(S)= \{r_0< r_1 < \ldots\}$, so $x = \sum _{m\in S} \frac{1}{u_{m}} = \sum _{n=0}^\infty \frac{1}{u_{r_n}}$. 
Let $C\in\N$ be an upper bound for all $b_n$, $n \in S$. 
For every $n\in \N$, since $r_n < r_{n+1} -1$, one has  
 $$u_{r_n-1} x \equiv_\Z  \frac{1}{b_{r_n}} +  \frac{1}{b_{r_n} \ldots b_{r_{n+1}}}+ \frac{1}{b_{r_n} \ldots b_{r_{n+1}} \ldots b_{r_{n+2}}}+
\ldots \leq \frac{1}{2} + \frac{1}{4} + \frac{1}{4^2} + \ldots = \frac{1}{2} + \frac{1}{3} = \frac{5}{6}.$$ 
As obviously also $\frac{1}{C}  \leq \frac{1}{b_{r_n}}  \leq \{u_{r_{n-1}} x\}$ holds, we deduce that  $\|u_{r_{n-1}} x\| \geq \frac{1}{C} $ for all $n\in \N_+$.  The sequence $u_n\bar x$ $\I$-converges to $0$, therefore $S -1 = \{r_n-1: n\in \N\}  \in \I$, and hence $S\in\I$ being $\I$ translation invariant.
\end{proof}

\subsection{The $\flat$-truncation and its atomic components}

Let $\uu\in\A$ and $x\in[0,1)$. Define the {\em $\flat$-truncation} of $x = \sum_{n\in \supp(x)} \frac{c_n(x)}{u_n} $ by 
\begin{equation}\label{Eq:May12}
{x^\flat} := \sum_{n\in \supp_b(x)} \frac{b_n-1}{u_n},
\end{equation}
i.e., $c_n(x^\flat) = c_n(x) = b_n-1$ for all $n\in \supp_b(x)$ and $c_n(x^\flat) = 0$ for all $n\in\N_+\setminus \supp_b(x)$. 

\smallskip
 If $\supp_b(x)$ is finite, then $\bar x^\flat \in  t_\uu^\I(\T)$. Hence, if $\supp_b(x)\subseteq_\I\supp(x)$ (i.e., $\iix$ holds), then $\supp(x)\in\I$ and so $\bar x\in  t_\uu^\I(\T)$ as well by Lemma~\ref{Lemma2.2}. So, now we are going to give a closer look at $x^\flat$ when $\supp_b(x)$ is infinite, associating to $x^\flat$ two atomic elements $\alpha^{(l)}(x)$ and $\alpha^{(r)}(x)$ as follows.

\begin{remark}\label{NEW:Rem}
Let $\uu\in\A$ and $x\in[0,1)$ such that $\supp(x)$ is neither finite nor cofinite, and also $\supp_b(x)$ is infinite. 
Then $x^\flat = \sum_{n=1}^\infty y_n$, where the elements $y_n= y_n^\flat$ have pairwise disjoint finite supports  which are  intervals $[l_n,r_n]_\N$.
So, by~\eqref{finiteestimation},
$$y_n= \frac{1}{u_{l_n-1}}-\frac{1}{u_{r_n}}.$$ 
Therefore, since the series defining $x^\flat$ is absolutely convergent, %we can write 
$$x^\flat = \sum_{n=1}^\infty y_n  = \sum_{n=1}^\infty \left( \frac{1}{u_{l_n-1}} - \frac{1}{u_{r_n}}\right)= 
 \sum_{n=1}^\infty  \frac{1}{u_{l_n-1}} - \sum_{n=1}^\infty \frac{1}{u_{r_n}} = \alpha^{(l)}(x) - \alpha^{(r)}(x)$$
 where we briefly write $\alpha^{(l)}(x):= \sum_{n=1}^\infty  \frac{1}{u_{l_n-1}}$ and $\alpha^{(r)}(x):= \sum_{n=1}^\infty \frac{1}{u_{r_n}}$;
this pair of atomic elements can be called {\em left and right atomic components of} $x^\flat$.
\end{remark}

The following is then a consequence of Proposition~\ref{Propo:New}.

\begin{corollary}\label{123}{\rm [$\ix\&\iiix\Rightarrow \overline{x^\flat}\in t_\uu^\I(\T)$]}
Let $\uu\in\A$, let $\I$ be a translation invariant free ideal of $\N$ and $x\in[0,1)$ with $S=\supp(x)$ infinite non-cofinite and $S_b=\supp_b(x)$ infinite.
\begin{itemize}
\item[(1)] If $\rho(S)\in\I$, then $\overline{\alpha^{(r)}(x)}\in t_\uu^\I(\T)$.
\item[(2)] If $\lambda(S)\in\I$, then $\overline{\alpha^{(l)}(x)}\in t_\uu^\I(\T)$.
\item[(3)] If $\rho(S)\cup\lambda(S)\in\I$ (i.e., $\ix\&\iiix$ holds), then $\overline{x^\flat} \in t_\uu^\I(\T)$.
\end{itemize}
\end{corollary}
\begin{proof}
(1) Since $\rho(S)=\supp(\alpha^{(r)}(x))$, $\rho(S)\in\I$ implies $\overline{\alpha^{(r)}(x)}\in t_\uu^\I(\T)$ by Proposition~\ref{Propo:New}.

(2) Since $\lambda(S)-1=\supp(\alpha^{(l)}(x))$, $\lambda(S)\in\I$ implies $\overline{\alpha^{(l)}(x)}\in t_\uu^\I(\T)$ by Proposition~\ref{Propo:New}.

(3) By (1) and (2), $\rho(S)\cup\lambda(S)\in\I$ implies $\overline{\alpha^{(r)}(x)}\in t_\uu^\I(\T)$ and $\overline{\alpha^{(l)}(x)}\in t_\uu^\I(\T)$, so $\overline{x^\flat}=\overline{\alpha^{(l)}(x)}-\overline{\alpha^{(r)}(x)}\in t_\uu^\I(\T)$.
\end{proof}

Under the assumption that $\rho(S)$ and $\lambda(S)-1$ are $b$-bounded, the implications in (1) and (2) of Corollary~\ref{123} can be inverted:

\begin{lemma}\label{lambdarho}
Let $\uu\in\A$, let $\I$ be a translation invariant free ideal of $\N$ and $x\in[0,1)$ with $S=\supp(x)$ infinite non-cofinite and $S_b=\supp_b(x)$ infinite.
  \begin{itemize}
\item[(1)] If $\rho(S)$ is $b$-bounded, then $\overline{\alpha^{(r)}(x)}\in t_\uu^\I(\T)$ (i.e., $\mathrm{(i}_{\alpha^{(r)}(x)}\mathrm{)}$ holds) if and only if $\rho(S)\in \I$ (i.e., $\ix$ holds); 
\item[(2)] If $\lambda(S)-1$ is $b$-bounded, then $\overline{\alpha^{(l)}(x)}\in t_\uu^\I(\T)$ (i.e., $\mathrm{(iii}_{\alpha^{(l)}(x)}\mathrm{)}$ holds) if and only if $\lambda(S)\in \I$ (i.e., $\iiix$ holds).
\end{itemize}
\end{lemma}
\begin{proof}
(1) If $\rho(S)\in\I$, then $\overline{\alpha^{(r)}(x)}\in t_\uu^\I(\T)$ by Lemma~\ref{Lemma2.2}

 Assume that $\rho(S)$ is $b$-bounded and $\overline{\alpha^{(r)}(x)}\in t_\uu^\I(\T)$; then $\rho(S)\in \I$ by Proposition~\ref{Propo:New}.

(2) If $\lambda(S)\in\I$, then $\overline{\alpha^{(l)}(x)}\in t_\uu^\I(\T)$ by Lemma  \ref{Lemma2.2}. Assume that $\lambda(S)-1$ is $b$-bounded and $\overline{\alpha^{(l)}(x)}\in t_\uu^\I(\T)$; then $\lambda(S)\in \I$ by Proposition~\ref{Propo:New}, as $\supp(\alpha^{(l)}(x) ) = \lambda(S)-1$ and $\I$ is translation invariant.
\end{proof}
 
\subsection{The strength of $\Tx$}

Let $\uu\in\A$, let $\I$ be a free ideal of $\N$ and $x,y\in[0,1)$. We write $x\equiv_\I y$ in case $\{n\in\N_+:c_n(x)\neq c_n(y)\}\in\I$. 
If $x\equiv_\I y$, then necessarily $\supp(x)=^\I\supp(y)$. In the specific case of $x$ and $x^\flat$, $x\equiv_\I x^\flat$ precisely when $\supp_b(x) \subseteq_\I \supp(x)$.

\begin{remark}\label{Iae} Let $\uu\in\A$, let $\I$ be translation invariant free ideal of $\N$ and $x,y\in[0,1)$. In this paper we use the following equivalences.
If $\supp(x)=^\I\supp(y)$, then $\ix\Leftrightarrow\mathrm{(i}_y\mathrm{)}$,  $\iix\Leftrightarrow\mathrm{(ii}_y\mathrm{)}$, $\IIx\Leftrightarrow\mathrm{(II}_y\mathrm{)}$ and $\axdue\Leftrightarrow\mathrm{(a2}_y\mathrm{)}$. If $x\equiv_\I y$, then also $\Ix\Leftrightarrow\mathrm{(I}_y\mathrm{)}$.
\end{remark}

\begin{proposition}\label{Coro:May11}
Let $\uu\in\mathcal A$, let $\I$ be a translation invariant free ideal of $\N$ and $x,y \in [0,1)$. If $x\equiv_\I y$, then $\bar{x}\in t_\uu^\I(\T)$ if and only if $\bar{y}\in t_\uu^\I(\T)$.
%In particular, if $x\equiv_\I x^\flat$, then $\bar{x}\in t_\uu^\I(\T)$ if and only if $\overline{x^\flat} \in t_\uu^\I(\T)$.
\end{proposition}
\begin{proof} 
Let $z= \sum_{n \in \supp(x) \cap \supp(y)}\frac{c_n(x)}{u_n}$, $x'=x - z$ and $y'= y - z$. Then, by hypothesis, $\supp(x') = \supp(x)\setminus \supp(y) \in \I$
and $\supp(y') = \supp(y)\setminus \supp(x) \in \I$, so $\overline{x'}, \overline{y'} \in t_\uu^\I(\T)$, by Lemma~\ref{Lemma2.2}. Therefore, $\bar x\in t_\uu^\I(\T)$ implies 
$\bar z\in t_\uu^\I(\T)$, which in turn implies $\bar y = \bar z + \overline{y'} \in t_\uu^\I(\T)$. Similarly, $\bar y \in t_\uu^\I(\T)$ implies $\bar x \in t_\uu^\I(\T)$. 
%
%Now assume that $\supp(x)\setminus \supp_b(x)\in\I$. Then $x$ and $x^\flat$ are $\I$-almost equal, so $\bar{x}\in t_\uu^\I(\T)$ if and only if  $\overline{x^\flat}\in t_\uu^\I(\T)$ by the first statement.
\end{proof}

%The following provides sufficient conditions for $\bar x \in t_\uu^\I(\T)$, without imposing any restraint on the sequence $\uu\in\A$.

 As an easy application of Proposition~\ref{Coro:May11}, we get:

\begin{corollary}\label{Coro1:May14}
Let $\uu\in\mathcal A$, let $\I\not= \F in$ be a free ideal of $\N$ and $x \in [0,1)$ with $S_b=\supp_b(x)$.
If $S_b\in \I^*$, then $x\equiv_\I x^\flat$ and $\overline{x^\flat}\in t_\uu^\I(\T)$, so $\bar{x}\in t_\uu^\I(\T)$.
\end{corollary}
\begin{proof} 
The hypothesis $S_b \in \I^*$ means that $\N \setminus S_b \in \I$, so $S \setminus S_b \in \I$ as well, and hence $x\equiv_\I x^\flat$.
Now by Proposition~\ref{Coro:May11}, $\overline{x^\flat} \in t_\uu^\I(\T)$ if and only if $\bar x\in t_\uu^\I(\T)$. 

It remains to show that $\overline{x^\flat}\in t_\uu^\I(\T)$. This is obvious when $S_b$ is finite, so assume that $S_b$ is infinite (so $S_b^*$ is not cofinite).
Let $x^* = \sum_{n\in S_b^*}\frac{b_n-1}{u_n}\in [0,1)$. Since $S_b^*$ is not cofinite, this is a representation as in~\eqref{ex-4}. Hence, Lemma~\ref{Lemma2.2} implies that $\overline{x^*} \in t_\uu^\I(\T)$. Since $x^\flat + x^*=1$, one gets $\overline{x^\flat} = -\overline{x^*} \in  t_\uu^\I(\T)$.
\end{proof}

In general $\supp(x) \in \I^*$ need not imply $\bar x \in t_\uu^\I(\T)$.
Indeed, if $\I$ is a proper free ideal of $\N$, $\uu=(3^n)$ and $x=\sum_{i=1}^\infty \frac{1}{3^n}$, then $u_nx\equiv_\Z \frac{1}{2}$ for every $n\in\N$.
Hence, $\bar x\not\in t_\uu^\I(\T)$, while $\supp(x)\in\I^*$, as $\supp(x)=\N_+$.

\begin{remark}\label{auto} 
Let $\uu\in\mathcal A$, let $\I$ be a translation invariant free ideal of $\N$ and $x\in[0,1)$. If $\supp(x)\in\I$, then $\Tx$ is automatically satisfied. 
\end{remark}

 If $\I$ is a translation invariant free ideal of $\N$, $\uu\in\A$ and $x\in[0,1)$ with $S=\supp(x)$, then $S-1\subseteq^\I S$ is equivalent to $S\subseteq^\I S+ 1$. Therefore, the conjunction of $\iiix$ with $\ix$ gives the important $\I$-equality $S \buildrel{\I}\over = S+1$, i.e., $S \Delta (S+1) \in \I$. Moreover, adding also $\iix$, we obtain the following:
\begin{claim} 
For $\uu\in\A$, a translation invariant free ideal $\I$ of $\N$ and $x\in[0,1)$ with $S=\supp(x)$ and $S_b=\supp_b(x)$, $\Tx$ is equivalent to $S + 1  \buildrel{\I}\over = S  \buildrel{\I}\over = S_b \ (\buildrel{\I}\over = S-1)$.
\end{claim}
\begin{proof} 
The equivalence is due to the fact that $S  \buildrel{\I}\over = S_b$ is equivalent to $\iix$, while $S + 1  \buildrel{\I}\over = S $  is equivalent to $\ix\&\iiix$.
\end{proof} 

\begin{theorem}\label{in}{\rm [$\Tx\Rightarrow\bar{x}\in t_\uu^\I(\T)$]}
Let $\uu\in\mathcal A$, let $\I$ be a translation invariant free ideal of $\N$ and $x\in[0,1)$.
If $\Tx$ holds, then $\bar{x}\in t_\uu^\I(\T)$.
\end{theorem}
\begin{proof} 
Assume that $x\in[0,1)$, with $S=\supp(x)$ and $S_b=\supp_b(x)$, satisfies  $\ix \&\iix \&\iiix$. In view of $\iix$ one has $x \equiv_\I x^\flat$, hence $\bar x \in t_\uu^\I(\T)$ if and only if $\overline{x^\flat} \in t_\uu^\I(\T)$ by Proposition~\ref{Coro:May11}. So, it remains to verify that $\overline{x^\flat} \in t_\uu^\I(\T)$. This is clear when $S$ is finite. When $S$ is cofinite, $S_b\subseteq_\I\N$, that is, $S_b\in\I^*$, and so $\overline{x^\flat} \in t_\uu^\I(\T)$ by Corollary~\ref{Coro1:May14}.
If $S$ is neither finite nor cofinite, $\overline{x^\flat}\in t_\uu^\I(\T)$ follows from $\ix \&\iiix$ by Corollary~\ref{123}(3).
\end{proof}

The next example shows that the implication $\bar{x}\in t_\uu^\I(\T) \Rightarrow \Tx $ fails in the strongest possible way with $\I=\F in$.

%\begin{example}\label{Exa:New}\NB
%Let $\uu = (n!)$, $\I = \F in$ and $x = \sum_{n=1}^\infty 1/(2n)!\in[0,1)$. Then $S = 2\N$, so $S$ and $S + 1$ are disjoint,
%and $S$ and $S - 1$ are disjoint, while $S_b = \{2\}$. Therefore, all three $(i_x), (ii_x)$ and $(iii_x)$ fail simultaneously, while $x \in t_\uu(\T)$, as $\lim\limits_n 1/n =0$. 
%\end{example}

%The following example shows that the implication $\bar x\in t_\uu^\I(\T)\Rightarrow \ix\&\iix$ in Lemma~\ref{Necessity2.12*} does not hold true without the hypothesis that $\supp(x)$ is $b$-bounded.
%Analogously, the implication $\bar x\in t_\uu^\I(\T)\Rightarrow \iiix$ from Theorem~\ref{3i} does not hold in general, and so also $\bar x\in t_\uu^\I(\T)\Rightarrow\Tx$ from Corollary~\ref{b-bounded}.

\begin{example}\label{ce}{\rm[$\bar{x}\in t_\uu^\I(\T) \not\Rightarrow \Tx$]}
Let $\uu=((n+1)!)$ and $x=\sum_{n=1}^\infty \frac{1}{(2n)!}$.
Then $S=\supp(x)=2\N+1$,  $(S+1)\cap S=\emptyset$ and $S_b = \emptyset$, so $\ix$, $\iix$ and $\iiix$ simultaneously fail. On the other hand, $\bar x\in t_\uu(\T)$ as $\lim\limits_{n\to\infty}\varphi\left(\frac{c_n}{n+1}\right)=\lim\limits_{n\in 2\N+1}\varphi\left(\frac{1}{n+1}\right)\to 0$, by Theorem~\ref{Prob:A}.
\end{example}

In view of Lemma~\ref{Lemma2.2} and Corollary~\ref{Coro1:May14},  for $x\in [0,1)$ with $\supp(x)= \supp_b(x)$, if either $\supp(x) \in \I$ or $\supp(x)^* \in \I$, then $\bar x \in t_\uu^\I(\T)$. This motivates the question of whether the converse implication holds true 
%\maltese\footnote{Ho aggiunto "when  $\supp(x)$ is $b$-bounded"}
when  $\supp(x)$ is $b$-bounded (it holds true for $\I=\F in$ by~\cite[Corollary~3.2]{DI}). The next example shows that this implication fails even for $\I = \I_d$.

\begin{example}\label{NoWC}
Let $(g_n)$ be the sequence of perfect squares (i.e., $g_n = n^2$ for $n\in \N$) and put $X= \bigcup_{n\in\N} [g_{2n}, g_{2n+1}]_\N$, so $X^*=  \bigcup_{n\in\N_+} (g_{2n-1}, g_{2n})_\N$.
We verify first that $d(X) = d(X^*)=1/2> 0$. To this end pick $k_n = \lfloor \frac{\sqrt{n}-1}{2}\rfloor\in \N$ (i.e., 
such that $(2k_n+1)^2 \leq n < (2k_n+3)^2$) for every $n \in \N_+$. 
Hence, 
\begin{equation}\label{L1}
\lim\limits_{n\to\infty}\frac{n}{(2k_n+1)^2}=1 \quad \text{and}\quad \lim\limits_{n\to\infty}\frac{n}{(2k_n+3)^2}=1.
\end{equation}
The choice of $k_n$ gives $|X((2k_n+1)^2)| \leq |X(n)| < |X((2k_n+3)^2)|$, hence
\begin{equation}\label{E2}
\frac{(2k_n+1)^2}{n} \cdot \frac{|X((2k_n+1)^2)|}{(2k_n+1)^2} \leq  \frac{|X(n)|}{n} <   \frac{(2k_n+3)^2}{n} \cdot \frac{|X((2k_n+3)^2)|}{(2k_n+3)^2}.
\end{equation}
Since, $|X((2k_n+1)^2)| =\sum_{i=0}^{k_n} [(2i+1)^2-(2i)^2+1] =\sum_{i=0}^{k_n}(4i+2)=2k_n^2+4k_n+2,$ we deduce that 
\begin{equation*}
%\label{L3}
\lim\limits_{n\to\infty}\frac{|X((2k_n+1)^2)|}{(2k_n+1)^2}=\lim\limits_{n\to\infty}\frac{2k_n^2+4k_n+2}{(2k_n+1)^2} =\frac{1}{2}. 
\end{equation*}
Similarly, $\lim\limits_{n\to\infty}\frac{|X((2k_n+3)^2)|}{(2k_n+3)^2}=\frac{1}{2}$. By these two limits and~\eqref{L1}, the first and the third member of the inequality~\eqref{E2} tend to $1/2$. This gives $d(X)=d(X^*)=1/2$, so $X\not \in \I_d$ and  $X^* \not \in \I_d$. 

Clearly, also $S=X\cap\N_+$ satisfies $d(S)=d(S^*)=1/2$. Fix arbitrarily a $b$-bounded $\uu\in\mathcal A$ 
%(e.g.,  $\uu = (2^n)$) 
and let $x\in [0,1)$ with $\supp(x) = \supp_b(x) = S$. Clearly, $\iix$ holds. Moreover, $\ix$ holds, as $\rho(S)= \{g_{2n+1}: n \in \N\} $, so $d(\rho(S))= 0$. Similarly, $\iiix$ holds, as $\lambda(S)=\{g_{2n}: n \in \N\}$, so $d(\lambda(S))=0$.

Therefore $\Tx$ holds, hence $\bar x \in t_\uu^{\I_d}(\T)$ by Theorem~\ref{in}, while $S \not  \in \I$ and $S^* \not  \in \I$.
\end{example}

\section{When $\supp(x)$ is $b$-bounded}\label{suppbbounded}

%In \S\ref{necessity} we mainly prove that under the assumption that $\supp(x)$ is $b$-bounded, $\ix$ and $\iix$ are necessary conditions for $\bar x$ to be a topologically $\uu_\I$-torsion element.
%Then, in \S\ref{allsupports} we find a characterization of the topologically $\uu_\I$-torsion elements with $b$-bounded support as those with $\Ax$.
We start this section by proving that under the assumption that $\supp(x)$ is $b$-bounded, $\ix$ and $\iix$ are necessary conditions for $\bar x$ to be a topologically $\uu_\I$-torsion element. Then, we find a characterization of the topologically $\uu_\I$-torsion elements with $b$-bounded support as those with $\Ax$.

%\smallskip
%Now, we show that the necessity in~\cite[Corollary 2.12]{Ghosh} holds true, namely, $\bar x\in t_\uu^\I(\T)  \Rightarrow \ix\&\iix$ when $\supp(x)$ is $b$-bounded:

\begin{lemma}\label{Necessity2.12*}
{\rm [$\bar x\in t_\uu^\I(\T)  \Rightarrow \axuno 
%\Leftrightarrow\axunop 
\Leftrightarrow \ix\&\iix$]} 
Let $\uu\in\mathcal A$, let $\I$ be a translation invariant free ideal of $\N$ and $x\in[0,1)$ with $S=\supp(x)$ and $S_b=\supp_b(x)$.
If $S$ is $b$-bounded, then $\ix\&\iix\Leftrightarrow\axuno$. In particular, $\bar x\in t_\uu^\I(\T)  \Rightarrow \ix\&\iix$. 
 \end{lemma}
\begin{proof}
By Proposition~\ref{primo}, $\ix\&\iix\Rightarrow \axuno$. To prove the converse implication $\axuno\Rightarrow\ix\&\iix$,  assume $\axuno$. Lemma~\ref{rem:nested*} shows that $\ix$ holds. We have to prove $\iix$, that is, $S_b\subseteq_\I S$. This is clear when $S\in\I$. If $S\not\in\I$, %and $S$ is $b$-bounded, 
then $\axuno$ with $A=S$ implies immediately that $S_b\subseteq_\I S$.

For the last assertion, $\bar x\in t_\uu^\I(\T)$ implies $\ax$, by Theorem~\ref{conjecture}. Then the first part of the lemma applies.
\end{proof}

If $S\in\I$, then $b$-boundedness of $S$ is not necessary in Lemma~\ref{Necessity2.12*}, since in this case all properties involved ($\ix\&\iix,\axuno$ and $\bar x\in t_\uu^\I(\T)$) are vacuously true. 
Moreover, Example~\ref{ce} shows that the implication in the final assertion of the lemma cannot be inverted in general. 

\smallskip
Before proving Theorem~\ref{Nuovo:Th}, it is worth recalling from Remark~\ref{auto} that if $\supp(x)\in\I$, then $\Tx$ is automatically satisfied.  Moreover, if $\supp(x)\in\I$ is $b$-bounded, also $\axuno$ and $\axdue$ are obviously satisfied. 
%and they are equivalent to $\axuno$ and $\axdue$, respectively, by Lemma~\ref{Necessity2.12*} and Proposition~\ref{primo}(2).

\begin{corollary}\label{Nuovo:Th-old}
Let $\uu\in\A$, $\I$ a translation invariant free $P$-ideal of $\N$ and $x\in[0,1)$. If $S=\supp(x)$ is $b$-bounded, then $\bar x\in t_\uu^\I(\T)$ if and only if $\Ax$ holds.
\end{corollary}
\begin{proof}
By Theorem~\ref{conjecture}, $\bar x\in t_\uu^\I(\T)$ is equivalent to $\ax=\axuno\&\axdue$, since $\bx$ is automatically satisfied under the assumption that $S$ is $b$-bounded. By Lemma~\ref{Necessity2.12*}, $\axuno$ is equivalent to $\ix\&\iix$. % while $\axdue$ is equivalent to $\axduep$ in view of Proposition~\ref{primo}(2).
\end{proof}

%%%%%%%%%%%%%%%
%The following is the proof of Theorem~\ref{Nuovo:Th}, i.e., for $\uu \in \A$ and a translation invariant free $P$-ideal of $\N$, if $X \in [0,1)$ with $\supp(x)$ $b$-bounded, then $\bar x\in t_\uu^\I(\T)$ if and only if $\Ax$ holds.

\begin{proof}[\bf Proof of Theorem~\ref{Nuovo:Th}]
 We have to prove that if $\uu \in \A$ and $\I$ is a translation invariant free $P$-ideal of $\N$ and $x \in [0,1)$ with $S=\supp(x)$ $b$-bounded mod $\I$, then $\bar x\in t_\uu^\I(\T)$ if and only if $\Ax$ holds.

By our hypothesis, there exists a $b$-bounded set $A\subseteq_\I S$. Let $y=\sum_{n\in A}\frac{c_n}{u_n}$, hence $\supp(y)$ is $b$-bounded and $x\equiv_\I y$.
Then $\bar x\in t_\uu^\I(\T)$ is equivalent to $\bar y\in t_\uu^\I(\T)$ by Proposition~\ref{Coro:May11}. Now, Corollary~\ref{Nuovo:Th-old} implies that the latter is equivalent to $\mathtt{A}_y$. It remains to recall that  $\mathtt{A}_y$ and $\Ax$ are equivalent, in view of Remark~\ref{Iae}.
\end{proof}

%In Proposition~\ref{3i} we see the role of  $b$-boundedness of $\Supp(x)$ in finding the equivalence of $\iiix$ with $\axdue$.
%This should be compared with Lemma~\ref{Necessity2.12*}, where $b$-boundedness of $S$ suffices to get the equivalence of $\ix\&\iix$ with $\axuno$, and with the following lemma.

%\NB\footnote{era stato detto male: Now we see that assuming that $\lambda(S)-1$ is $b$-bounded, $\axdue$ becomes equivalent to $\iiix$, so also $\Ax$ and $\Tx$.}
Now we see that $\axdue \Leftrightarrow \iiix$ and $\Ax\Leftrightarrow\Tx$ under the assumption  that $\lambda(S)-1$ is $b$-bounded. 

\begin{proposition}\label{3i} {\rm[$\axdue \Leftrightarrow \iiix$]} 
Let $\uu\in\mathcal A$, let $\I$ be a translation invariant free ideal of $\N$ and let $x\in[0,1)$ with $S=\supp(x)$. If $\lambda(S)-1$ is $b$-bounded, then $\axdue \Leftrightarrow \iiix$, and so $\Ax\Leftrightarrow\Tx$. 
\end{proposition}
\begin{proof} By Proposition~\ref{primo}, $\iiix\Rightarrow\axdue$, so in order to prove that  $\Ax\Leftrightarrow\Tx$, it is enough to check that $\axdue\Rightarrow \iiix$, and this is verified in Lemma~\ref{lambda-1} under the assumption that $\lambda(S)-1$ is $b$-bounded.
%To prove the last assertion recall that if $S$ is $b$-bounded, Proposition~\ref{primo}(2) gives that $\axdue\Leftrightarrow\axduep$. Hence, we can conclude using also the implication proved above.
\end{proof}

The equivalence $\bar x\in t_\uu^\I(\T)\Leftrightarrow \Tx$ can be achieved under a stronger assumption of $b$-boundedness (compared to Theorem~\ref{Nuovo:Th}):

\begin{corollary}\label{b-bounded}{\rm[$\bar x\in t_\uu^\I(\T)\Leftrightarrow\ax \Leftrightarrow \Tx$]} 
Let $\uu\in\A$, let $\I$ be a translation invariant free $P$-ideal of $\N$ and $x\in[0,1)$. 
If $S\sqcup(\lambda(S)-1)$ is $b$-bounded, then $\bar x\in t_\uu^\I(\T)$ if and only if $\Tx$ holds. 
\end{corollary}
\begin{proof}
By Theorem~\ref{Nuovo:Th}, $\bar x\in t_\uu^\I(\T)$ is equivalent to $\Ax$. Now, $\Ax$ is equivalent to $\Tx$ in view of Proposition~\ref{3i}.

A second proof can be given without making recurse to Theorem~\ref{Nuovo:Th}.
Indeed, that $\bar x\in t_\uu^\I(\T)$ implies $\ax$ comes from Theorem~\ref{conjecture}, while $\ax\Rightarrow\Tx$ is given by Lemma~\ref{Necessity2.12*} and Proposition~\ref{3i}, and $\Tx\Rightarrow\bar x\in t_\uu^\I(\T)$ by Theorem~\ref{in}. 
\end{proof}

In particular the equivalence from Corollary~\ref{b-bounded} holds for $b$-bounded arithmetic sequences:

\begin{corollary}\label{b-boundedreal}
Let $\uu\in\A$, let $\I$ be a translation invariant free $P$-ideal of $\N$ and $x\in[0,1)$. If $\uu$ is $b$-bounded, then $\bar x\in t_\uu^\I(\T)$ if and only if $\Tx$ holds.
\end{corollary}

\section{Nested ideals}\label{nested}

 In this section we introduce nested ideals and see that all ideals $\I_\alpha$ have this property. On the other hand, in \S\ref{Countersec} we find examples of translation invariant analytic free $P$-ideals that are not nested, providing counterexamples to~\cite[Corollary 2.12]{Ghosh}.
Then, in \S\ref{Key} we prove that under the additional hypothesis on the ideal to be nested, the conclusion of~\cite[Corollary 2.12]{Ghosh} holds true (see Corollary~\ref{CoroGhosh}).

\smallskip
 Remark~\ref{rem:nested} motivates the following: 

\begin{definition}\label{nesteddef} 
A pair $((l_n),(r_n))$ of strictly increasing sequences in $\N$ is said to be {\em left nested} if  \eqref{(*)} holds, namely, if $l_n\leq r_n<l_{n+1}-1$ for every $n\in \N$. 
A free ideal $\I$ of $\N$ is said to be {\em nested} whenever for any left nested pair  $((l_n),(r_n))$ of sequences in $\N$, if $R=\{r_n:n\in\N\}\in \I$, then $L=\{l_n:n\in\N\} \in \I$.  
\end{definition}

\begin{example}\label{DenId:is:nested} 
Clearly, $\F in$ is a nested translation invariant analytic free $P$-ideal.

Moreover, for every $\alpha\in(0,1]$, also the ideal $\I_\alpha$  is a translation invariant analytic free $P$-ideal that is nested. This follows immediately from the fact that if $((l_n),(r_n))$ is a pair of left nested sequences in $\N$, and $L=\{l_n:n\in\N\}$, $R=\{r_n:n\in\N\}$, then $|L(l_n)|\leq |R(l_n)|+1$ for every $n\in\N$ (in fact, $|L(l_n)|=|R(l_n)|$ if $l_n=r_n$ and $|L(l_n)=|R(l_n)|+1$ if $l_n<r_n$).
\end{example}

In particular, by Example~\ref{DenId:is:nested} there are $\mathfrak c$ many nested translation invariant analytic free $P$-ideals of $\N$. By a result of Solecki~\cite{Sol2} there are $\mathfrak c$ many analytic ideals of $\N$.  On the other hand, we expect that a more careful rewriting of Example~\ref{counterexample*} would provide $\mathfrak c$ many non-nested translation invariant analytic free $P$-ideals of $\N$.

\smallskip
 The next is a useful equivalent condition for a translation invariant free ideal to be nested.

\begin{lemma}\label{equi}
A  translation invariant free ideal $\I$ of $\N$ is nested if and only if for every $S\subseteq \N$, $\rho(S)\in \I$ implies $\lambda(S)\in \I$.
\end{lemma}
\begin{proof} 
Assume that $\I$ is nested and pick an arbitrary $S\subseteq \N$.  If $S$ is finite or cofinite, then $\lambda(S)$ is finite and therefore belongs to $\I$. 
Otherwise, $\lambda(S)$ and $\rho(S)$ are infinite, so considered both as increasing one-to-one sequences they form a left nested pair, by Remark~\ref{rem:nested}.
The hypothesis $\rho(S)\in \I$ yields $\lambda(S)\in \I$, as $\I$ is a nested ideal. 

Vice versa, assume that for every $S\subseteq \N$, $\rho(S)\in \I$ implies $\lambda(S)\in \I$. To check that $\I$ is a nested ideal, pick a left nested pair $((l_n),(r_n))$ of sequences in $\N$, let $L=\{l_n: n\in\N\}$, $R=\{r_n:n\in\N\}$ and assume that $R\in \I$. By Remark~\ref{rem:nested}, there exists $S\subseteq \N$ having $\lambda(S)=L$ and $\rho(S)=R\in\I$. Therefore, $L=\lambda(S) \in\I$ by hypothesis. Hence, $\I$ is nested.
\end{proof}
%\begin{proof} 
%Assume that $\I$ is nested and pick an arbitrary $S\subseteq \N$.  If $S$ is finite or cofinite, then $\lambda(S)$ is finite and therefore belongs to $\I$. 
%Otherwise, the pair $\lambda(S)$, $\rho(S)$ is a left nested pair, by Remark~\ref{rem:nested}.
%The hypothesis $\rho(S)\in \I$ yields $\lambda(S)\in \I$, as $\I$ is a nested ideal. 
%
%Vice versa, assume that for every $S\subseteq \N$, $\rho(S)\in \I$ implies $\lambda(S)\in \I$. 
%To check that $\I$ is a nested ideal, pick a left nested pair $L$, $R$ of sets in $\N$ with $R\in \I$. By Remark~\ref{rem:nested}, there exists $S\subseteq \N$ having $\lambda(S) \subseteq ^*L$ and $\rho(S)  \subseteq ^* R$. Then $\rho(S) \in\I$. Therefore, $\lambda(S) \in\I$ by hypothesis, and thus $L\in\I$, as $\lambda(S) \subseteq ^*L$ and $\I$ is free.
%\end{proof}

So, for $\uu\in\A$, a nested translation invariant free ideal $\I$ of $\N$ and $x\in[0,1)$, $\ix\Rightarrow\iiix$; hence, $\ix\&\iix\equiv\Tx$.

\subsection{Examples of translation invariant analytic free $P$-ideals that are not nested}\label{Countersec}
%\maltese\footnote{In questa subsection ho sostituito i $\gamma$ con i $\boldsymbol{\gamma}$}
In Example~\ref{counterexample*} we see that there are plenty of translation invariant analytic free $P$-ideals that are not nested.

\smallskip
We need the following notions. A \emph{submeasure} on $\N$ is a function $\varphi\colon \P(\N) \to \R_{\geq 0}\cup \{\infty\}$ such that $\varphi(\emptyset)=0$, and for all $A,B\in\P(\N)$, $\varphi(A \cup B) \leq \varphi(A)+\varphi(B)$ and $A \subseteq B$ implies $\varphi(A) \leq  \varphi(B)$.  
Moreover, a function $\varphi\colon M\to \R$, where $M$ is a topological space, is \emph{lower semicontinuous} if $x_n\to x$ in $M$ implies $\liminf_{n\to\infty}\varphi(x_n)\geq\varphi(x)$.

For any lower semicontinuous submeasure $\varphi$, one defines the \emph{exhaustive ideal} 
as follows
$$\mathrm{Exh}(\varphi):=\left\{A \subseteq \N: \lim\limits_{n\to\infty} \varphi(A \setminus[0,n]_{\N})=0\right\}.$$
For a sequence $\boldsymbol{\gamma}=(\gamma_n)$ of strictly positive real numbers with $\sum_{n\in\N}\gamma_n=\infty$, define 
$$\mu_{\boldsymbol{\gamma}}\colon\P(\N)\to\R_{\geq 0}\cup\{\infty\}, \quad \mu_{\boldsymbol{\gamma}}(A):=\sum_{n\in A}\gamma_n\ \text{for every $A\subseteq \N$}.$$
Let also $\I_{\boldsymbol{\gamma}}:=\{A\subseteq \N: \mu_{\boldsymbol{\gamma}}(A)<\infty\}.$
Note that $\I_{\boldsymbol{\gamma}}= \I_{\boldsymbol{\gamma'}}$ whenever the sequences ${\boldsymbol{\gamma}}$ and ${\boldsymbol{\gamma'}}$ eventually coincide (i.e., $\I_{\boldsymbol{\gamma}}$ does not vary if one alters finitely many members of the sequences ${\boldsymbol{\gamma}}$).

\begin{lemma}\label{ti2}
Let ${\boldsymbol{\gamma}}=(\gamma_n)$ be a sequence of strictly positive real numbers with $\sum_{n\in\N}\gamma_n=\infty$. 
Then:
\begin{itemize}
\item[(1)] $\I_{\boldsymbol{\gamma}}$ is an analytic free $P$-ideal;
\item[(2)] if $\frac{1}{2}\leq \frac{\gamma_{n+1}}{\gamma_n} \leq 2$ for every $n\in\N_+$, then $\I_{\boldsymbol{\gamma}}$ is translation invariant (hence, unboundedly convex).
\end{itemize}
\end{lemma}
\begin{proof} 
(1) Clearly, $\mu_{\boldsymbol{\gamma}}$ is a lower semicontinuos submeasure, as $\mu_{\boldsymbol{\gamma}}$ is a submeasure and for all $A\in\mathcal P(\N)$, $\mu_{\boldsymbol{\gamma}}(A)=\lim\limits_{n\to\infty}\mu_{\boldsymbol{\gamma}}(A(n))$. Hence, $\mathrm{Exh}(\mu_{\boldsymbol{\gamma}})$ is an $F_{\sigma\delta}$-set and a $P$-ideal by~\cite[Lemma~1.2.2]{Fa}, therefore an analytic free $P$-ideal.
Moreover, $\mathrm{Exh}(\mu_{\boldsymbol{\gamma}})=\I_{\boldsymbol{\gamma}}$, as $\mu_{\boldsymbol{\gamma}}(A) < \infty$ if and only if $ \lim\limits_{n\to\infty} \mu_{\boldsymbol{\gamma}}(A \setminus[0,n]_{\N})=0$.

\smallskip
(2) It is easy to verify that for every $A\in\I_{\boldsymbol{\gamma}}$, also $A+1,A-1\in\I_{\boldsymbol{\gamma}}$. This implies that $\I_{\boldsymbol{\gamma}}$ is translation invariant.
\end{proof}

It is easy to see that the restraint $\frac{1}{2}\leq \frac{\gamma_{n+1}}{\gamma_n} \leq 2$ in item (2) can be relaxed to 
$c \leq \frac{\gamma_{n+1}}{\gamma_n} \leq C$, for any pair of  arbitrary positive constants $c< C$. 

\begin{example}\label{counterexample*} 
In order to define the sequence ${\boldsymbol{\gamma}}=(\gamma_n)_{n\in \N_+}$, fix a real number $q\in (1/2, 1)$ and 
consider a sequence ${\boldsymbol{\alpha}}=(\alpha_n)_{n\in\N}$ of positive real numbers and  such that
$\frac{1}{2}\leq \frac{\alpha_{n+1}}{\alpha_n} \leq q$. 
This ensures  
\begin{equation}\label{Eq:15Sept}
\sum_{n\in\N} \alpha_n\leq \sum_{n\in\N} q^n < \infty. 
\end{equation}
 Next we define a partition $\N_+=  \bigsqcup_{n=0}^\infty B_n$ into intervals $B_n$, defined as follows. Let  
\begin{equation}\label{Eq:18Sept}
w_n = n^2 + 1, \  z_n=n^2 + n +1 \mbox{ and } B_n =  [w_n, w_{n+1}-1]_\N  \mbox{ for each $n\in\N$},
\end{equation}
so that 
\begin{equation}\label{Eq:18Sept*}
w_n = n^2 + 1 < z_n = n^2 + n +1 < (n+1)^2 = w_{n+1} -1\quad\text{and}\quad z_n \in B_n\ \text{for each}\ n\in\N_+.
\end{equation}
Moreover,  $B_0 = \{1\}$ is degenerate and $\N_+=  \bigsqcup_{n=0}^\infty B_n$.
The range of the sequence ${\boldsymbol{\gamma}}$ is the set $A:= \{\alpha_n: n \in \N_+\}$ in $\R$ and the function $\gamma\colon \N_+ \to A$, $n\mapsto\gamma_n$, is (piecewise) defined as follows. For $n\in \N$,  the restriction of $\gamma$ to the interval $B_n$ is defined by 
$$ \gamma_m = \begin{cases} \alpha_{n-m+w_{n}},\ \text{if}\ w_{n}\leq m\leq z_n\\ \alpha_{m-n-w_{n}},\ \text{if}\ z_n\leq m\leq w_{n+1}-1 \!\!\!\!\end{cases}\!\!, \ \  m\in  B_n = [w_n, w_{n+1}-1]_\N. $$
In particular, $\gamma_1 =\alpha_0$ as $B_0= \{1\}$ is degenerate. Put $W=\{w_n:n\in\N\}$ and $Z=\{z_n:n\in\N\}$.
Roughly speaking, the graph of $\gamma\colon \N_+ \to A$ is ``wave-like", where the local minima are given by the restriction $\gamma\restriction_W\colon W \to A$, defined by $\gamma(w_n) = \alpha_{n}$ for every $n\in\N$. The values of all local maxima are equal to $1$, attained  precisely in $Z$. Moreover, $\gamma$ is increasing on the first half $[w_{n}, z_n]_\N$ of $B_n$ and decreasing on its second half $[z_n, w_{n+1}-1]_\N$. 

The following diagram, where, for the sake of brevity, we put $k_n=w_{n+1}-1$ for every $n\in\N$, may help to understand better the sequence ${\boldsymbol{\gamma}}$: 
\begin{align*}
&\underbrace{\alpha_0=1}_{=\gamma_{w_0}=\gamma_{z_0}=\gamma_{k_0}},\\
\underbrace{\alpha_1}_{=\gamma_{w_1}},& \underbrace{\alpha_0=1}_{=\gamma_{z_1}},\underbrace{\alpha_1}_{=\gamma_{k_1}},\\
\underbrace{\alpha_2}_{=\gamma_{w_2}},\alpha_1,
&\underbrace{\alpha_0=1}_{=\gamma_{z_2}}, \alpha_1, \underbrace{\alpha_2}_{=\gamma_{k_2}},\\ 
& \vdots\\
\underbrace{\alpha_n}_{=\gamma_{w_n}}, \alpha_{n-1},\dots,\alpha_2,\alpha_{1}, &\underbrace{\alpha_0=1}_{=\gamma_{z_n}}, \alpha_1,\alpha_2,\dots,\alpha_{n-1},\underbrace{\alpha_n}_{=\gamma{k_n}}, \\
\underbrace{\alpha_{n+1}}_{=\gamma_{w_{n+1}}},\dots, \dots,  \dots, \dots, \dots,  & \underbrace{\alpha_0=1}_{=\gamma_{z_{n+1}}}, \dots, \dots,  \dots, \dots, \dots, \underbrace{\alpha_{n+1}}_{=\gamma_{k_{n+1}}}\\
&\vdots
\end{align*}

(a) As $\sum_{n\in\N}\gamma_n=\infty$ and $\frac{1}{2}\leq \frac{\gamma_{n+1}}{\gamma_n} \leq 2$ for every $n\in\N$, the ideal $\I_{\boldsymbol{\gamma}}$ is a translation invariant analytic free $P$-ideal by Lemma~\ref{ti2}.

\smallskip
(b) The pair $((z_n),(w_{n+1}))$ is left nested  in view of~\eqref{Eq:18Sept*} (as $z_n \leq w_{n+1} < z_{n+1} - 1$ for all $n\in \N_+$). Using also~\eqref{Eq:15Sept}, we obtain.
$$\mu_{\boldsymbol{\gamma}}(W)=\sum_{n\in\N}\gamma_{w_n}=\sum_{n\in\N}\alpha_n\leq \sum_{n\in\N}q^{n}
< \infty,\quad\text{while}\quad \mu_{\boldsymbol{\gamma}}(Z)=\sum_{n\in\N}\gamma_{z_n}=\sum_{n\in\N}1=\infty.$$ 
Hence, $W\setminus\{w_0\}=\{w_{n+1}:n\in\N\}\in\I_{\boldsymbol{\gamma}}$ and $Z=\{z_n:n\in\N\}\notin\I_{\boldsymbol{\gamma}}$ witness that $\I_{\boldsymbol{\gamma}}$ is not nested. 
\end{example} 

As an application, we get the following example of a sequence $\uu\in \A$, a proper  translation invariant analytic free $P$-ideal $\I$ of $\N$ and $x\in [0,1)$ 
with $b$-bounded support, such that  $\Ax$ holds, but $\Tx$ fails. 
In particular, this shows that $\bar x\in t_\uu^\I(\T)$ does not imply $\Tx$ even under the assumption that $\supp(x)$ is $b$-bounded.

\begin{example}\label{Exa:2:osserv} {\rm[$\Ax\not\Rightarrow\Tx$]}
Let $\I=\I_{\boldsymbol{\gamma}}$ be a proper translation invariant analytic free $P$-ideal of $\N$ which is not nested. 

By Lemma~\ref{equi}, there is a set $S\subseteq \N$ such that $\rho(S)\in\I$ and $\lambda(S)\not\in\I$. Since $S$ contains $\lambda(S)$, so $S\notin \I$. Similarly, 
using that $S^* \supseteq \lambda(S)-1\notin\I$ as $\lambda(S)\not\in\I$ and $\I$ is translation inviariant, we deduce that  $S^*\notin\I$. In particular, both $S$ and $S^*$ are infinite.

In order to define $\uu\in\A$, for every $n\in S$, let $b_n=2$, and for every $n\in S^*$, let $b_n=2^n$. Furthermore, let $x=\sum_{i\in S}\frac{1}{u_i}$. Then $S=\supp(x)=\supp_b(x)$ is $b$-bounded, so $\iix$ is satisfied. Moreover, $\ix$ is satisfied while $\iiix$ (hence $\Tx$ as well) fails, by our choice of $S$. 

Next we verify that $\axdue$ is satisfied, and hence also $\Ax$.
Let $A\in\P(\N)\setminus \I$ be $b$-bounded with $A\cap S\in\I$.
Since $S^*$ is $b$-divergent, $A\cap S^*$ must be finite. Hence, $A = (A\cap S)\cup (A\cap S^*) \in \I$. Thus, 
$A+1\in \I$ as well and consequently,  $(A+1) \cap S \in \I$. Therefore, $\axdue$ holds.
\end{example}

\subsection{The key characterization of nested ideals via $t_\uu^\I(\T)$}\label{Key}

The next proposition shows that adding the hypothesis nested to the translations invariant free $P$-ideal $\I$, $\ix$ and $\iix$ alone suffice to deduce that $\bar x \in t^{\I}_\uu(\T)$ without any additional hypothesis on $\supp(x)$. Then Corollary~\ref{CoroGhosh} proves that~\cite[Corollary 2.12]{Ghosh}, holds when $\I$ is also nested.

\begin{proposition}\label{cor} 
Let $\uu\in\mathcal A$, let $\I$ be a translation invariant free nested $P$-ideal of $\N$ and $x\in[0,1)$. Then $\ix\&\iix \Rightarrow \bar x \in t^{\I}_\uu(\T)$.
\end{proposition}
\begin{proof}
Let $S = \supp(x)$. By Lemma~\ref{equi}, $\ix$ implies $\iiix$ since $\I$ is a nested ideal. 
In conjunction with $\iix$ this covers the hypothesis of Theorem~\ref{in}, so we conclude that $\bar x \in t^{\I}_\uu(\T)$. 
\end{proof}

\begin{corollary}\label{CoroGhosh} 
Let $\uu\in\mathcal A$, let $\I$ be a translation invariant free $P$-ideal of $\N$ and $x\in[0,1)$.
\begin{itemize}
\item[(1)] If $\supp(x)$ is $b$-bounded and $\bar x \in t^{\I}_\uu(\T)$, then $\ix\&\iix$ holds.
\item[(2)] If $\ix\&\iix$ hold and $\I$  is nested, then $\bar x \in t^{\I}_\uu(\T)$. 
\end{itemize}
\end{corollary}
\begin{proof} 
(1) was checked in Lemma~\ref{Necessity2.12*}, while (2) is Proposition~\ref{cor}.
\end{proof}

\begin{remark}\label{New:Rem}
By Example~\ref{DenId:is:nested} every $\I_{\alpha}$ with $\alpha\in(0,1]$ is nested, hence the above corollary also gives a complete proof of~\cite[Corollary 3.9]{DG1} (see Corollary~\ref{DG13.9}).
 
On the other hand, Proposition~\ref{notnested->} provides for every  $b$-bounded $\uu\in \A$ and every non-nested translation invariant free ideal $\I$ of $\N$ an element $x\in [0,1)$ with $\bar x\not \in t_\uu^\I(\T)$ and satisfying both $\ix$ and $\iix$. 
%(which is equivalent to $\neg \Tx$ due to the above corollary as $\uu\in \A$ is $b$-bounded). 
In particular, this shows that Corollary~\ref{CoroGhosh}(2) fails for non-nested ideals and for every $b$-bounded sequence $\uu\in \A$. 

Therefore,~\cite[Corollary 2.12]{Ghosh} is false, as the hypothesis of nestedness is missing there and Example~\ref{counterexample*} produces translation invariant analytic free $P$-ideals that are not nested.
\end{remark}

\begin{proposition}\label{notnested->}
{\rm [$\ix\&\iix\not\Rightarrow\iiix$, $\ix\&\iix\not\Rightarrow \bar x\in t_\uu^\I(\T)$]} Let $\uu\in \A$ be $b$-bounded and let $\I$ be a translation invariant free ideal of $\N$. If $\I$ is not nested, then there exists $x\in t_\uu^\I(\T)$ such that $\ix\&\iix$ holds but $\iiix$ does not hold. So also $\bar x\not\in t_\uu^\I(\T)$.
\end{proposition}
\begin{proof} 
Let $((l_n),(r_n)$ be the pair of left nested sequences in $\N$ witnessing that $\I$ is not nested, namely, $L=\{l_n:n\in\N\}\not\in\I$ and $R=\{r_n:n\in\N\}\in\I$. Define $S=\bigcup_{n\in\N}[l_n,r_n]_\N$ and $x=\sum_{i\in S}\frac{b_i-1}{u_i}$.  Then $\ix$ holds, since $R=\rho(S)\in\I$. As $\supp(x)=\supp_b(x)$, also $\iix$ holds. Moreover, $L=\lambda(S)\not\in\I$, hence $\iiix$ does not hold. In view of Corollary~\ref{b-boundedreal}, $\bar x\not\in t_\uu^\I(\T)$.
\end{proof}

\begin{proof}[\bf Proof of Theorem~\ref{Cor:May12prec}]
The implications (2)$\Rightarrow$(3)$\Rightarrow$(4)$\Rightarrow$(5) are trivial, (5)$\Rightarrow$(1) is Proposition~\ref{notnested->} and (1)$\Rightarrow$(2) is Proposition~\ref{cor}.
\end{proof}

The implications (4)$\Rightarrow$(5)$\Rightarrow$(1)$\Rightarrow$(2)  do not use the assumption that $\I$ is a $P$-ideal. As far as  the implication (1)$\Rightarrow$(2) is concerned, we note that in (2) the implication~\eqref{Eq:May23} is required for {all} $x\in[0,1)$, not only for those with $b$-bounded support (see Proposition~\ref{cor}).

\section{When $\supp(x)$ is $b$-divergent}\label{Sec:Coro3.10:DG1}

Here we consider the case when $\supp(x)$ is $b$-divergent and we give a proof to Theorem~\ref{Last:corollary}, which is inspired by~\cite[Corollary 3.4]{DI} (see Corollary~\ref{unbounded}) and generalizes it. 
Moreover, Theorem~\ref{Last:corollary} specialized for $\I=\I_\alpha$ is~\cite[Corollary~3.10]{DG1} (see Corollary~\ref{C3.10}).
While the proof of~\cite[Corollary 3.10]{DG1} uses the counterpart of Theorem~\ref{conjecture} for the sufficiency of the conditions $\Ix$ and $\IIx$, the proof of Theorem~\ref{Last:corollary} does not make recurse to Theorem~\ref{conjecture} for this implication.

\begin{proposition}\label{III}
Let $\uu\in\A$, let $\I$ be a translation invariant free ideal of $\N$ and $x\in[0,1)$ with $S=\supp(x)$ $b$-divergent. Then 
$\ax\Leftrightarrow\axdue\Leftrightarrow\IIx$ and $\bx\Leftrightarrow\Ix$.
\end{proposition}
\begin{proof} 
$\axdue\Rightarrow\IIx$
Let $D\in\P(\N)\setminus \I$ with $D\subseteq^\I S$ and $D-1$ $b$-bounded. Since $D\not\in\I$, also $D-1\not\in\I$ being $\I$ translation invariant. As $S$ is $b$-divergent and $D-1$ is $b$-bounded, $(D-1)\cap S$ is finite and in particular $(D-1)\cap S\in\I$. By $\axdue$, there exists $A'\subseteq_\I D-1$ such that $\lim\limits_{n\in A'}\frac{c_{n+1}}{b_{n+1}}=0$. Consequently, $D'=A'+1\subseteq_\I D$ and $\lim\limits_{n\in D'}\frac{c_n}{b_n}=0$. Namely, $\IIx$ holds.

$\IIx\Rightarrow\ax$
Assume that $A\in\P(\N)\setminus\I$ is $b$-bounded. 
Since $S$ is $b$-divergent and $A$ is $b$-bounded, $A\cap S$ is finite and so $A\cap S\in\I$ while $A\setminus S\not\in\I$. Therefore, it suffices to verify $\axdue$. 

Suppose that $(A+1)\cap S\in\I$. There exists $A'\subseteq\N$ such that  $A'+1=(A+1)\cap S^*$. Then $A'+1 \subseteq_\I A+1$ and $\lim\limits_{n\in A'+1}\frac{c_{n}}{b_{n}}=0$ since $c_{n}=0$ for every $n\in A'+1$. In particular, $A'\subseteq_\I A$ and $\lim\limits_{n\in A'}\frac{c_{n+1}}{b_{n+1}}=0$.

Now assume that $D:=(A+1)\cap S\not\in\I$.
Since $D-1\subseteq A$ is $b$-bounded, by $\IIx$ there exists $D'\subseteq_\I D$ such that $\lim\limits_{n\in D'}\frac{c_n}{b_n}=0$. 
Let $D''=D'\sqcup(A+1)\cap S^*\subseteq_\I A+1$. As $c_n=0$ for every $n\in(A+1)\cap S^*$, we get that $\lim\limits_{n\in D''}\frac{c_n}{b_n}=0$.
For $A'=D''-1\subseteq_\I A$, $\lim\limits_{n\in A'}\frac{c_{n+1}}{b_{n+1}}=0$, %Since $A\cap S$ is finite, let $B'=A'\setminus S$
so $\axdue$ holds.

$\bx\Rightarrow\Ix$ When $S\not\in\I$, then $\bx$ applied to $S$ gives $\Ix$. 

$\Ix\Rightarrow\bx$ If $S\in\I$, then $\bx$ holds. So assume that $S\not\in\I$ and that $A\in\P(\N)\setminus\I$ is $b$-divergent. If $A\cap S\in\I$, then $A\setminus S\subseteq_\I A$ and $\lim\limits_{n\in A\setminus S}\varphi\left(\frac{c_n}{b_n}\right)=0$ since $c_n=0$ for every $n\in A\setminus S$. Now assume that $A\cap S\not\in\I$. By $\Ix$, there exists $D\subseteq_\I S$ such that $\lim\limits_{n\in D}\varphi\left(\frac{c_n}{b_n}\right)=0$. Let $B:=(D\cap A)\sqcup (A\setminus S)$, where $D\cap A\subseteq_\I A\cap S$. Then $B\subseteq_\I A$ and $\lim\limits_{n\in B}\varphi\left(\frac{c_n}{b_n}\right)=0$. So, $\bx$ holds.
\end{proof}

\begin{remark}\label{notIx}
Here we discuss the validity of $\Ix$ and $\IIx$ for $\uu\in\A$, a translation invariant free ideal $\I$ of $\N$ and $x\in[0,1)$ when $\supp(x)\in\I$.
Obviously, the assumption $\supp(x)\in\I$ makes $\IIx$ vacuously true. 
As far as $\Ix$ is concerned, the statement ``$\mathrm{(\#)}$ there exists $D\subseteq_\I S$ such that $\lim\limits_{n\in D}\varphi\left(\frac{c_n}{b_n}\right)=0$'' obviously fails when $\supp(x)$ is $b$-bounded. 
Nevertheless, this non-implication is valid also when $\I$ contains some (necessarily infinite) $S_0\in \I$ which is $b$-divergent. 
Indeed,  let $x\in[0,1)$ with $\supp(x)=S_0$ and $c_n=\lfloor\frac{b_n}{2}\rfloor$ for every $n\in S_0$; then $\mathrm{(\#)}$ fails for this $x$, 
as $\lim\limits_{n\in S_0}\frac{c_n}{b_n}=\frac{1}{2}$, so if $D\subseteq S_0$ and $D$ is infinite, again $\lim\limits_{n\in D}\frac{c_n}{b_n}=\frac{1}{2}$. 
\end{remark}

The next implication of Theorem~\ref{Last:corollary} follows also from Proposition~\ref{III} and Theorem~\ref{conjecture}; nevertheless, we provide a direct simple proof, since the proof of the sufficiency in Theorem~\ref{conjecture} is quite heavy.

\begin{proposition}\label{salvataggio}
Let $\uu\in\A$, let $\I$ be a translation invariant free ideal of $\N$ and $x\in[0,1)$ with $S=\supp(x)$ $b$-divergent. If $\Ix\&\IIx$ holds, then $\bar x\in t_\uu^\I(\T)$.
\end{proposition}
\begin{proof}
 If $S \in\I$, then $\bar x\in t_\uu^\I(\T)$ by Lemma~\ref{Lemma2.2}. So, assume that $S\not\in\I$. In view of $\Ix$, there exists $D\subseteq_\I S$ such that $\lim\limits_{n\in D}\varphi\left(\frac{c_n}{b_n}\right)=0$. Let $y=\sum_{n\in D}\frac{c_n}{u_n}$, then $x\equiv_\I y$, as $D\subseteq_\I S$. By Remark~\ref{Iae}, $\Ix\Leftrightarrow\mathrm{(I}_y\mathrm{)}$ and $\IIx\Leftrightarrow\mathrm{(II}_y\mathrm{)}$. Moreover, $\bar x\in t_\uu^\I(\T)$ if and only if $\bar y\in t_\uu^\I(\T)$. This is why we can assume that $x=y$ (and consequently, $S=D$) from now on, so that $\lim\limits_{n\to\infty}\varphi\left(\frac{c_n}{b_n}\right)=\lim\limits_{n\in S}\varphi\left(\frac{c_n}{b_n}\right) =0$.
Therefore, it remains to see that $\bar x \in t_\uu^\I(\T)$.

\smallskip
Let $0<\eps<\frac{1}{4}$ and let $k_0\in\N_+$ with $\left\Vert \frac{c_k}{b_k}\right\Vert<\eps$ for every $k\geq k_0$ and $b_k\geq\frac{1}{\eps}$ for every $k\in S$ with $k\geq k_0$. To this end it is sufficient to  prove that $E:=\{m>k_0:\Vert u_mx\Vert\geq 3\eps\}\in\I$.

First we see that 
\begin{equation}\label{ES}
\text{if $m\in E$, then $m+1\not\in S$.}
\end{equation}
In fact, let $m\in E$ and assume for a contradiction that $m+1\in S$. Using~\eqref{NEW:Lemma},
$$u_mx\equiv_\Z u_m\frac{c_{m+1}}{u_{m+1}}+u_m\sum_{i=m+2}^\infty\frac{c_i}{u_i}\leq \frac{c_{m+1}}{b_{m+1}}+\frac{1}{b_{m+1}}$$
and $m+1\in S$,  one obtains
$3\eps\leq\Vert u_m x\Vert\leq\left\Vert \frac{c_{m+1}}{b_{m+1}}\right\Vert+\frac{1}{b_{m+1}}\leq2\eps,$
a contradiction.

\smallskip
For $m\in E$, let $k_m=\min\{k\in S:k>m\}$ and note that $k_m\geq m+2$ by~\eqref{ES}. Let $K=\{k_m:m\in E\}$. Let $r\in\N_+$ with $2^{-r}\leq 2\eps$. Now we prove the following assertion, which gives the required $E\in\I$:
\begin{equation}\label{goal}
E\subseteq\bigcup_{i=2}^{r+1} (K-i)\quad\text{and}\quad K\in\I.
\end{equation}

Let $m\in E$ and $k=k_m$.
By the choice of $k$, 
$$
u_mx\equiv_\Z u_m\frac{c_k}{u_k}+u_m\sum_{i=k+1}^\infty \frac{c_i}{u_i}, \quad \text{where} \quad u_m\sum_{i=k+1}^\infty \frac{c_i}{u_i}\leq \frac{u_m}{u_k}\leq \frac{u_{k-1}}{u_k}=\frac{1}{b_k}<\eps.
$$
Hence, 
\begin{equation}\label{2eps}
\left\Vert u_m\frac{c_k}{u_k}\right\Vert\geq \Vert u_m x\Vert-\eps\geq 2\eps.
\end{equation}
It follows that 
$$\frac{1}{2^r}\leq 2\eps\leq u_m\frac{c_k}{u_k}=\frac{u_m}{u_{k-1}} \frac{c_k}{b_k}\leq \frac{u_m}{u_{k-1}}\leq \frac{1}{2^{(k-1)-m}}$$
and so
$r\geq k-1-m$, namely, $m+2\leq k\leq m+ r+1$. Therefore, there exists $i\in\N$ with $2\leq i\leq r+1$ such that $m=k-i$. This proves the inclusion $E\subseteq\bigcup_{i=2}^{r+1}(K-i)$ in~\eqref{goal}.

\smallskip
Now we verify that $K\in\I$ as announced in~\eqref{goal}, which implies that $\bigcup_{i=2}^{r+1}(K-i)\in\I$, as $\I$ is translation invariant, and so we can conclude that $E\in\I$, as required. 

First, we see that $K-1$ is $b$-bounded. Let $k\in K$, so there exists $m\in E$ with $k=k_m$. Since $u_m\frac{c_k}{u_k}\geq 2\eps$ in view of~\eqref{2eps} and $k\geq m+2$, 
\begin{equation}\label{*H}
2\eps\leq\frac{u_{k-2} c_k}{u_k}=\frac{1}{b_{k-1}}\frac{c_k}{b_k}\leq \frac{1}{b_{k-1}},
\end{equation}
hence $b_{k-1}\leq \frac{1}{2\eps}$.
This shows that $K-1$ is $b$-bounded.

Now assume for a contradiction that $K\not\in\I$. As $K \subseteq S$, we can apply $\IIx$. This entails the existence of $k\in K$ such that $\frac{c_k}{b_k}\leq\eps$. Using~\eqref{*H},
$$
2\eps\leq\frac{1}{b_{k-1}}\frac{c_k}{b_k}\leq \frac{1}{b_{k-1}}\eps \quad \text{implies that} \quad b_{k-1}\leq\frac{1}{2},$$ which is a contradiction. So, $K\in\I$ and we conclude the proof.
\end{proof}

\begin{corollary}\label{Last:corollary-old}
Let $\uu\in\A$, let $\I$ be a translation invariant free $P$-ideal of $\N$ and $x\in[0,1)$ with $\supp(x)$ $b$-divergent. Then $\bar x\in t_\uu^\I(\T)$ if and only if $\Ix\&\IIx$ holds.
\end{corollary}
\begin{proof}
That $\Ix\&\IIx$ implies $\bar x\in t_\uu^\I(\T)$ follows from Proposition~\ref{salvataggio}. The converse implication follows from Theorem~\ref{conjecture} (where the hypothesis $P$-ideal is needed) and Proposition~\ref{III}.
\end{proof}

\begin{proof}[\bf Proof of Theorem~\ref{Last:corollary}] 
We have to prove that if $\uu \in \A$ and $\I$ is a translation invariant free $P$-ideal of $\N$ and $x \in [0,1)$ with $\supp(x)$ $b$-divergent mod $\I$, then $\bar x\in t_\uu^\I(\T)$ if and only if $\Ix\&\IIx$ holds.

Let $M\subseteq_\I\N$ be $b$-divergent. Moreover, let $y=\sum_{n\in M}\frac{c_n}{u_n}$, so that $x\equiv_\I y$. Then $\bar x\in t_\uu^\I(\T)$ is equivalent to $\bar y\in t_\uu^\I(\T)$ by Proposition~\ref{Coro:May11}. Now, in view of Corollary~\ref{Last:corollary-old}, this is equivalent to $\mathrm{(I}_y\mathrm{)}\&\mathrm{(II}_y\mathrm{)}$. By Remark~\ref{Iae}, $\Ix$ is equivalent to $\mathrm{(I}_y\mathrm{)}$ and $\mathrm{(II}_x\mathrm{)}$ is equivalent to $\mathrm{(II}_y\mathrm{)}$.
\end{proof}

Next we see in Corollary~\ref{NmodI}, that under the strong hypothesis that the whole $\N$ is $b$-divergent mod $\I$, $\Ix$ is sufficient to get $\bar x\in t_\uu^\I(\T)$.
First, the next lemma shows the strength of the hypothesis.

\begin{lemma}\label{Bu}
Let $\uu\in\A$ and let $\I$ be a free ideal of $\N$. Then $\N$ is $b$-divergent mod $\I$ if and only if every $b$-bounded subset of $\N$ is contained in $\I$.
\end{lemma}
\begin{proof}
Assume that $\N$ is $b$-divergent mod $\I$ and let $D\subseteq_\I \N$ be $b$-divergent; in particular, $D^*\in\I$. Let $B$ be a $b$-bounded subset of $\N$. Then $B\cap D$ is finite, so $B\subseteq^* D^*\in\I$, and hence $B\in\I$.
Vice versa, assume that all $b$-bounded subsets of $\N$ are in $\I$. 
%\maltese\footnote{Ho modificato il riferimento. Prima era \cite[Lemma~2.11]{I-torsion}}
By~\cite[Lemma~2.10]{I-torsion}, there exists a $b$-divergent $D\subseteq_\I\N$, that is, $\N$ is $b$-divergent mod $\I$.
\end{proof}

\begin{corollary}\label{NmodI}
Let $\uu\in\A$, let $\I$ be a translation invariant free $P$-ideal of $\N$ and $x\in[0,1)$ with $S=\supp(x)$. If $\N$ is $b$-divergent mod $\I$, then $\bar x\in t_\uu^\I(\T)$ if and only if $\Ix$ holds.
\end{corollary}
\begin{proof}
Every $b$-bounded subset of $\N$ is in $\I$ by Lemma~\ref{Bu}, hence $\IIx$ is automatically satisfied.
If $S\in\I$, then $\bar x\in t_\uu^\I(\T)$ by Lemma~\ref{Lemma2.2}, and $\Ix$ obviously holds. So, assume that $S\not\in\I$; since $\N$ is $b$-divergent mod $\I$, also $S$ is $b$-divergent mod $\I$. Hence, $\bar x\in t_\uu^\I(\T)$ is equivalent to $\Ix$ in view of Theorem~\ref{Last:corollary}.
\end{proof}

\section{Under the $\I$-splitting property}\label{Sec:I-spl}

Given $\uu\in\A$ and a free ideal $\I$, recall from the introduction that a subset $A$ of $\N$ has the \emph{$\I$-splitting property} if it has the form $A=B\sqcup D$, with $B$ $b$-bounded mod $\I$ and $D$ $b$-bounded mod $\I$. Moreover, we say that \emph{$\uu$ has the $\I$-splitting property} if $\N_+$ has the $\I$-splitting property.
Equivalently, as in \cite[Definition 3.1]{Ghosh}, $\uu$ has the $\I$-splitting property if there exists a partition $\N=B \sqcup D$ such that: 
\begin{itemize}
    \item[(1)] $B$ and $D$ are either empty or $B, D\not\in\I$;
    \item[(2)] if $B\not\in\I$, then $B$ is $b$-bounded mod $\I$;
    \item[(3)] if $D\not\in\I$, then $D$ is $b$-divergent mod $\I$. 
\end{itemize}
The dichotomy in item (1) is not restrictive as if $\N=B \sqcup D$ and $\emptyset\neq D\in\I$, then $B\subseteq_\I\N$, and analogously, $D\subseteq_\I\N$ when $B\in\I$.

%\begin{definition}[{\cite[Definition 3.1]{Ghosh}}]\label{Def:I-spl}
%Let $\I$ be a free ideal of $\N$. 
%A sequence $(b_n)$ of natural numbers 
%is said to have the {\em $\I$-splitting property} if there exists a partition $\N=B \sqcup D$ such that: 
%\begin{itemize}
%    \item[(1)] $B$ and $D$ are either empty or $B, D\not\in\I$;
%    \item[(2)] if $B\not\in\I$, then $B$ is $b$-bounded mod $\I$;
%    \item[(3)] if $D\not\in\I$, then $D$ is $b$-divergent mod $\I$. 
%\end{itemize}
%\end{definition}
%The dichotomy in item (1) is not restrictive as if $\N=B \sqcup D$ and $\emptyset\neq D\in\I$, then $B\subseteq_\I\N$, and analogously, $D\subseteq_\I\N$ when $B\in\I$.
%
%So, given $\uu\in\A$ and a free ideal $\I$ of $\N$, $\mathbf b$ has the $\I$-splitting property if and only if $\N_+$ has the $\I$-splitting property.

\begin{remark}
Let $\uu\in\A$, let $\I$ be a translation invariant free $P$-ideal of $\N$ and $x\in[0,1)$ with $S=\supp(x)$.
Assume that $S$ has the $\I$-splitting property, namely, $S=B\sqcup D$ with $B$ $b$-bounded mod $\I$ and $D$ $b$-divergent mod $\I$. 
Write $x=x_B+x_D$ with $$x_B=\sum_{n\in B}\frac{c_n}{u_n}\quad\text{and}\quad x_D=\sum_{n\in D}\frac{c_n}{u_n}.$$

Since we are looking for sufficient conditions to get $\bar x\in t_\uu^\I(\T)$, we observe that if simultaneously $\bar x_B\in t_\uu^\I(\T)$ and $\bar x_D\in t_\uu^\I(\T)$, then $\bar x\in t_\uu^\I(\T)$. We show below that this implication can be inverted. 
\end{remark}

In the following result we abbreviate the property $\Ix\&\IIx$ by $\mathtt D_x$.

\begin{theorem}\label{Lemma1}
Let $\uu\in\A$, let $\I$ be a translation invariant free $P$-ideal of $\N$ and $x\in[0,1)$ with $S=\supp(x)$.
Assume that $S$ has the $\I$-splitting property, namely, $S=B\sqcup D$ with $B$ $b$-bounded mod $\I$ and $D$ $b$-divergent mod $\I$. 
Write $x=x_B+x_D$ with $$x_B=\sum_{n\in B}\frac{c_n}{u_n}\quad\text{and}\quad x_D=\sum_{n\in D}\frac{c_n}{u_n}.$$
Consider the following properties:
\begin{itemize}
\item[(1)] $\mathtt A_{x_B}$ holds and either $D\in\I$ or $D\not\in\I$ and $\mathtt D_{x_D}$ holds;
\item[(2)] $\bar x_B\in t_\uu^\I(\T)$ and $\bar x_{D}\in t_\uu^\I(\T)$;
\item[(3)] $\bar x\in t_\uu^\I(\T)$;
\item[(4)] $\mathtt D_{x_D}$ holds;
\item[(5)] $\bar x_D\in t_\uu^\I(\T)$.
\end{itemize}
Then (1)$\Leftrightarrow$(2)$\Leftrightarrow$(3)$\Rightarrow$(4)$\Leftrightarrow$(5).
\end{theorem}
\begin{proof}
(1)$\Leftrightarrow$(2) Since $B$ is $b$-bounded mod $\I$, $\bar x_B\in t_\uu^\I(\T)$ is equivalent to $\mathtt A_{x_B}$ by Theorem~\ref{Nuovo:Th}. Since $D$ is $b$-divergent mod $\I$, $\bar x_D\in t_\uu^\I(\T)$ holds if and only if either $D\in\I$ or $D\not\in\I$ and $\mathtt D_{x_D}$ holds, by Theorem~\ref{Last:corollary}.

\smallskip
(2)$\Rightarrow$(3) is clear.

\smallskip
(3)$\Rightarrow$(4) Assume that $D\not\in\I$, let $D'\subseteq_\I D$ be $b$-divergent and let $x_{D'}=\sum_{n\in D'}\frac{c_n}{u_n}$. First we see that $\bx\Rightarrow\mathrm{(I}_{x_D}\mathrm{)}$: since $D=\supp(x_D)\not\in\I$, also $D'\not\in\I$. Then $\bx$ applied to $D'$ gives $\mathrm{(I}_{x_{D'}}\mathrm{)}$, and so also $\mathrm{(I}_{x_D}\mathrm{)}$ holds by Remark~\ref{Iae}.
Now we verify that $\axdue\Rightarrow\mathrm{(II}_{x_D}\mathrm{)}$. 
Let $E\in\P(\N)\setminus \I$ with $E\subseteq^\I D$ and $E-1$ $b$-bounded. Since $E\not\in\I$, also $E-1\not\in\I$ being $\I$ translation invariant. As $D$ is $b$-divergent  mod $\I$ and $E-1$ is $b$-bounded, $(E-1)\cap D\in\I$. By $\axdue$, there exists $A'\subseteq_\I E-1$ such that $\lim\limits_{n\in A'}\frac{c_{n+1}}{b_{n+1}}=0$. Consequently, $E'=A'+1\subseteq_\I E$ and $\lim\limits_{n\in E'}\frac{c_n}{b_n}=0$, namely, $\mathrm{(II}_{x_D}\mathrm{)}$ holds. 

\smallskip
(4)$\Leftrightarrow$(5) This equivalence is given by Theorem~\ref{Last:corollary}. 

\smallskip
(3)$\Rightarrow$(2) 
Assume that $\bar x\in t_\uu^\I(\T)$. If $D\in\I$, then $\bar x_{D}\in t_\uu^\I(\T)$ by Lemma~\ref{Lemma2.2}. If $D\not\in\I$, then $\bar x_D\in t_\uu^\I(\T)$ by the implication (3)$\Rightarrow$(4) and Theorem~\ref{Last:corollary}. In particular, $\bar x\in t_\uu^\I(\T)$ implies $\bar x_{D}\in t_\uu^\I(\T)$, and so $\bar x_{B} = \bar x - \bar x_{D}\in t_\uu^\I(\T)$ as well. 
\end{proof}

It is easy to verify that (4) does not imply (3) in general.

\begin{remark}\label{Lemma1**} 
Since in the sequel we will be interested mainly in the equivalence (1)$\Leftrightarrow$(3) from Theorem~\ref{Lemma1} (see Question~\ref{QuestionSplitting}) and in order to ease the reader, we formulate here explicitly item (1) of that theorem (saying ``$\mathtt A_{x_B}$ and $\mathtt D_{x_D}$ hold") in expanded form. Namely, the following conditions simultaneously hold, where $\supp(x_B)=B\cap S$, $\supp_b(x_B)=B\cap S_b$ and $\supp(x_D)=D\cap S$:
\begin{itemize}
\item[${\mathrm{(i}_{x_B}\mathrm{)}}$] $\supp(x_B)+1\subseteq^\I \supp(x_B)$; 
\item[${\mathrm{(ii}_{x_B}\mathrm{)}}$] $\supp_b(x_B)\subseteq_\I \supp(x_B)$;
\item[${\mathrm{(a2}_{x_B}\mathrm{)}}$] if $A\in \P(\N) \setminus\I$ is $b$-bounded  and $A\cap \supp(x_B) \in \I$, then $(A+1)\cap \supp(x_B)\in \I$;
%\end{itemize} 
%\noindent moreover, either $D\in\I$ or $D\not\in\I$ and: 
%\begin{itemize}
\item[${\mathrm{(I}_{x_D}\mathrm{)}}$]  either $D\in\I$ or $D\not\in\I$ and there exists $E\subseteq_\I \supp(x_D)$ such that $\lim\limits_{n\in E}\varphi\left(\frac{c_n}{b_n}\right)=0$;
\item[${\mathrm{(II}_{x_D}\mathrm{)}}$] for every $E\in\P(\N)\setminus\I$ with $E\subseteq^\I \supp(x_D)$ such that $E-1$ is $b$-bounded, there exists $E'\subseteq_\I E$ such that $\lim\limits_{n\in E'}\frac{c_n}{b_n}=0$.
\end{itemize}
\end{remark}

We recall the following result from~\cite{I-torsion}, as  we are interested in the subsequent question.

\begin{corollary}[\cite{I-torsion}]\label{ThGh:May29}  
Let $\uu\in\mathcal A$, let $\I$ a translation invariant free $P$-ideal of $\N$ and $x \in[0,1)$ with $S=\supp(x)$ and $S_b=\supp_b(x)$. If $\uu$ has the $\I$-splitting property  witnessed by the partition $\N=B\sqcup D$, then $\bar x \in t^{\I}_\uu(\T)$ if and only if the following conditions hold:
\begin{itemize}
   \item[$\unox$] $(B\cap S)+1 \subseteq^{\I} S$, $B\cap S \subseteq^{\I} S_b$ and if $B\cap S \not\in \I$ then there exists $C\subseteq_{\I} B\cap S$ such that $\lim\limits_{n \in C}\frac{c_{n+1}+1}{b_{n+1}}=1$; 
   \item[$\duex$] if $B\setminus S\not\in\I$, then there exists  $C \subseteq_{\I} B\setminus S$  such that $\lim\limits_{n \in C}\frac{c_{n+1}}{b_{n+1}}=0$;
   \item[$\trex$] if $D\cap S\not\in\I$, then there exists $E \subseteq_{\I} D\cap S$ such that $\lim\limits_{n \in E} \varphi\left(\frac{c_n}{b_n}\right)=0$.
\end{itemize}
\end{corollary}

\begin{question}\label{QuestionSplitting}
Does Theorem~\ref{Lemma1} imply Corollary~\ref{ThGh:May29} assuming that $\uu$ has the $\I$-splitting property?
\end{question}

In the following claim we find a partial answer to Question~\ref{QuestionSplitting}. More precisely, 
in Claim~\ref{TheClaim} we see that the conditions $\mathtt A_{x_B}$ and ${\mathrm{(I}_{x_D}\mathrm{)}}$ (without ${\mathrm{(II}_{x_D}\mathrm{)}}$) used in Theorem~\ref{Lemma1}(1) imply the condition $\unox\&\duex\&\trex$ used in Corollary~\ref{ThGh:May29}. 

\begin{claim}\label{TheClaim}
Let $\uu\in\mathcal A$, let $\I$ a translation invariant free $P$-ideal of $\N$ and $x \in[0,1)$ with $S=\supp(x)$ and $S_b=\supp_b(x)$. Assume that $\uu$ has the $\I$-splitting property witnessed by the partition $\N=B\sqcup D$ with $B$ $b$-bounded mod $\I$ and $D$ $b$-divergent mod $\I$, and let $x_B=\sum_{n\in B\cap S}\frac{c_n}{u_n}$ and $x_D=\sum_{n\in D\cap S}\frac{c_n}{u_n}$. Then $\mathtt T_{x_B}\& \mathrm{(I}_{x_D}\mathrm{)}$ implies $\unox\&\duex\&\trex$.
\end{claim}
\begin{proof}
In order to discuss in more detail the relation between the conditions in Theorem~\ref{Lemma1}(1) and those in Corollary~\ref{ThGh:May29}, let us call $\unox'$, $\unox''$ and $\unox'''$ the three conditions in $\unox$, namely,
\begin{itemize}
\item[$\unox'$] $(B\cap S)+1\subseteq^\I S$;
\item[$\unox''$] $B\cap S\subseteq^\I S_b$;
\item[$\unox'''$] if $B\cap S\not\in\I$, then there exists $C\subseteq_\I B\cap S$ such that $\lim\limits_{n\in C+1}\frac{c_n+1}{b_n}=1$.
\end{itemize}
 
${\mathrm{(i}_{x_B}\mathrm{)}}\Rightarrow \unox'$  Assume ${\mathrm{(i}_{x_B}\mathrm{)}}$, that is, $(B\cap S)+1\setminus(B\cap S)\in\I$. Since $(B\cap S)+1\setminus S\subseteq (B\cap S)+1\setminus (B\cap S)$, also $(B\cap S)+1\setminus  S\in\I$, that is, $\unox'$ holds.

\smallskip
${\mathrm{(ii}_{x_B}\mathrm{)}}\Leftrightarrow\unox''$ Clearly, ${\mathrm{(ii}_{x_B}\mathrm{)}}$ means that $(B\cap S)\setminus(B\cap S_b)\in\I$, and this is equivalent to $(B\cap S)\setminus S_b\in\I$, which is $\unox''$.

\smallskip
${\mathrm{(i}_{x_B}\mathrm{)}}\&{\mathrm{(ii}_{x_B}\mathrm{)}}\Rightarrow\unox'''$ By ${\mathrm{(ii}_{x_B}\mathrm{)}}$, $(B\cap S)\setminus(B\cap S_b)\in\I$, while ${\mathrm{(i}_{x_B}\mathrm{)}}$ means that $((B\cap S)+1)\setminus(B\cap S)\in\I$, hence, also $((B\cap S)+1)\setminus(B\cap S_b)\in\I$ as $B\cap S\subseteq_\I B\cap S_b$. Let $D=((B\cap S)+1)\cap(B\cap S_b)$ and $C=D-1$, so that $D\subseteq_\I (B\cap S)+1$ and $C\subseteq_\I B\cap S$.
Moreover, since $c_n=b_n-1$ for all $n\in D$ as $D\subseteq S_b$, we conclude that $\lim\limits_{n\in C+1}\frac{c_n+1}{b_n}=\lim\limits_{n\in D}\frac{b_n}{b_n}=1$.

\smallskip
${\mathrm{(a2}_{x_B}\mathrm{)}}\Rightarrow\duex$ Here ${\mathrm{(a2}_{x_B}\mathrm{)}}$ means that, if $A\in\P(\N)\setminus\I$ is $b$-bounded and $A\cap (B\cap S)\in\I$, then there exists $C\subseteq_\I A$ such that $\lim\limits_{n\in C+1}\frac{c_n}{b_n}=0$.

Assume that $B\setminus S\in\I$. Let $B_1\subseteq_\I B$ with $B_1$ $b$-bounded. Then also $B_1\setminus S\in\I$ and we apply ${\mathrm{(a2}_{x_B}\mathrm{)}}$ to $B_1\setminus S$. Then there exists $C\subseteq_\I B_1\setminus S$ such that $\lim\limits_{n\in C+1}\frac{c_n}{b_n}=0$. Clearly, $C\subseteq_\I B$.

\smallskip
The equivalence ${\mathrm{(I}_{x_D}\mathrm{)}}\Leftrightarrow\trex$ is clear.
\end{proof}

As a consequence of Claim~\ref{TheClaim} we find a partial answer to Question~\ref{QuestionSplitting}, i.e.,  by applying Theorem~\ref{Lemma1} and Claim~\ref{TheClaim}, we obtain the implication $\bar x \in t_\uu^\I(\T)\Rightarrow \unox\&\duex\&\trex$ in Corollary~\ref{ThGh:May29}. 
The question of whether the inverse implication $\unox\&\duex\&\trex \Rightarrow \bar x \in t_\uu^\I(\T)$
can be deduced from Theorem~\ref{Lemma1} under the assumption that  $\uu$ has the $\I$-splitting property remains still open. 
% rom Question~\ref{QuestionSplitting} it is open whether one can 

\smallskip
Another consequence of Claim~\ref{TheClaim} is the following corollary showing that 
the condition ${\mathrm{(II}_{x_D}\mathrm{)}}$ in Theorem~\ref{Lemma1}(1) can be relaxed under the stronger assumption that $\uu$ has the $\I$-splitting property.
 
\begin{corollary}\label{alleq*}
Let $\uu\in\mathcal A$, let $\I$ a translation invariant free $P$-ideal of $\N$ and $x \in[0,1)$ with $S=\supp(x)$ and $S_b=\supp_b(x)$. Assume that $\uu$ has the $\I$-splitting property  witnessed by the partition $\N=B\sqcup D$ and let $x=x_B+x_D$ with $x_B=\sum_{n\in B\cap S}\frac{c_n}{u_n}$ and  $x_D=\sum_{n\in D\cap S}\frac{c_n}{u_n}$. Then $\bar x \in t_\uu^\I(\T)$ if and only if $\mathtt A_{x_B}$ and ${\mathrm{(I}_{x_D}\mathrm{)}}$ hold.
\end{corollary}
\begin{proof}
Assume that $\mathtt A_{x_B}$ %holds and either $D\cap S\in\I$ or $D\cap S\not\in\I$ 
and ${\mathrm{(I}_{x_D}\mathrm{)}}$ hold.
By Claim~\ref{TheClaim}, this implies that $\unox\&\duex\&\trex$ holds, so $\bar x \in t_\uu^\I(\T)$ by Corollary~\ref{ThGh:May29}.
Vice versa, if $\bar x \in t_\uu^\I(\T)$, then $\mathtt A_{x_B}$ %holds and either $D\cap S\in\I$ or $D\cap S\not\in\I$ 
and ${\mathrm{(I}_{x_D}\mathrm{)}}$ hold by Theorem~\ref{Lemma1}.
\end{proof}

\end{document}